\documentclass[12pt]{amsart}
\usepackage{amsmath}
\usepackage{amsthm}
\usepackage{amssymb}
\usepackage{mathtools}
\usepackage{graphicx}
\usepackage{mathrsfs}
\usepackage{mdframed}
\usepackage[multiple,stable]{footmisc}
\usepackage{multicol}
\usepackage{enumerate}
\usepackage{tikz}
\usetikzlibrary{arrows}
\usepackage[pdfauthor={Neil Barton and Kameryn J. Williams},
    pdftitle={Varieties of class-theoretic potentialism},
    hidelinks 
]{hyperref}

\theoremstyle{plain}
\ifdefined\theorem \else \newtheorem{theorem}{Theorem} \fi
\ifdefined\maintheorem \else \newtheorem{maintheorem}[theorem]{Main Theorem} \fi
\ifdefined\theoremschema \else \newtheorem{theoremschema}[theorem]{Theorem Schema} \fi
\ifdefined\lemma \else \newtheorem{lemma}[theorem]{Lemma} \fi
\ifdefined\lemmaschema \else \newtheorem{lemmaschema}[theorem]{Lemma Schema} \fi
\ifdefined\proposition \else \newtheorem{proposition}[theorem]{Proposition} \fi
\ifdefined\corollary \else \newtheorem{corollary}[theorem]{Corollary} \fi
\ifdefined\fact \else \newtheorem{fact}[theorem]{Fact} \fi
\ifdefined\problem \else \newtheorem{problem}[theorem]{Problem} \fi
\ifdefined\conjecture \else \newtheorem{conjecture}[theorem]{Conjecture} \fi
\ifdefined\question \else \newtheorem{question}[theorem]{Question} \fi
\ifdefined\observation \else \newtheorem{observation}[theorem]{Observation} \fi
\newtheorem*{theorem*}{Theorem}
\newtheorem*{lemma*}{Lemma}
\newtheorem*{proposition*}{Proposition}
\newtheorem*{conjecture*}{Conjecture}
\newtheorem*{question*}{Question}
\newtheorem*{definition*}{Definition}

\theoremstyle{definition}
\ifdefined\definition \else \newtheorem{definition}[theorem]{Definition} \fi
\ifdefined\maindefinition \else \newtheorem{maindefinition}[theorem]{Main Definition} \fi
\theoremstyle{remark}
\ifdefined\remark \else \newtheorem{remark}[theorem]{Remark} \fi
\ifdefined\remarks \else \newtheorem{remarks}[theorem]{Remarks} \fi
\ifdefined\example \else \newtheorem{example}[theorem]{Example} \fi
\ifdefined\claim \else \newtheorem{claim}[theorem]{Claim} \fi
\ifdefined\exercise \else \newtheorem{exercise}[theorem]{Exercise} \fi
\newtheorem*{remark*}{Remark}
\ifdefined\acknowledgment \else \newtheorem*{acknowledgment}{Acknowledgment} \fi
\ifdefined\dedication \else \newtheorem*{dedication}{Dedicated} \fi
\ifdefined\case \else \newtheorem*{case}{Case} \fi

\newcounter{my_enumerate_counter}

\newcommand\comment[1]{}

\newcommand\Zsf{\mathsf{Z}}

\newcommand\Afrak{\mathfrak{A}}
\newcommand\Bfrak{\mathfrak{B}}

\newcommand\Tfrak{\mathfrak{T}}

\newcommand\Pcal{\mathcal{P}}

\newcommand\Xcal{\mathcal{X}}
\newcommand\Ycal{\mathcal{Y}}

\newcommand\Qbb{\mathbb{Q}}

\newcommand{\dom}{\operatorname{dom}}

\ifdefined\alt \else
  \newcommand{\alt}{\operatorname{alt}}
\fi
\newcommand{\crit}{\operatorname{crit}}

\newcommand\axiom{\mathsf}

\newcommand\ZF{\axiom{ZF}}
\newcommand\ZFC{\axiom{ZFC}}
\newcommand\ZFm{\axiom{ZF}^-}

\newcommand\GB{\axiom{GB}}
\newcommand\GBC{\axiom{GBC}}

\newcommand\ETR{\axiom{ETR}}

\newcommand\MK{\axiom{MK}}

\newcommand\AC{\axiom{AC}}

\newcommand\Sfour{\axiom{S4}}
\newcommand\Sfourtwo{\axiom{S4.2}}
\newcommand\Sfourthree{\axiom{S4.3}}

\newcommand\class{\mathrm}

\newcommand\HOD{\class{HOD}}
\newcommand\Ord{\class{Ord}}

\newcommand\Add{\class{Add}}

\newcommand{\seq}[1]{\left\langle #1 \right\rangle}

\renewcommand{\epsilon}{\varepsilon}

\newcommand\card[1]{\left\lvert #1 \right\rvert}

\newcommand\rest{\upharpoonright}

\newcommand\powerset{\Pcal}

\newcommand{\godel}[1]{\ulcorner#1\urcorner}

\newcommand\Con{\operatorname{Con}}

\newcommand\Def{\operatorname{Def}}
\ifdefined\Form \else
  \newcommand\Form{\axiom{Form}}
\fi

\newcommand{\impl}{\Rightarrow}
\renewcommand{\iff}{\Leftrightarrow}
\renewcommand{\phi}{\varphi}

\DeclareMathOperator{\possible}{\text{\tikz[scale=.6ex/1cm,baseline=-.6ex,rotate=45,line width=.1ex]{\draw (-1,-1) rectangle (1,1);}}}
\DeclareMathOperator{\necessary}{\text{\tikz[scale=.6ex/1cm,baseline=-.6ex,line width=.1ex]{\draw (-1,-1) rectangle (1,1);}}}

\newcommand{\Tr}{\mathrm{Tr}}

\renewenvironment{description}[1][0pt]
  {\list{}{\labelwidth=0pt \leftmargin=#1
   }}
  {\endlist}

\title{Varieties of Class-Theoretic Potentialism}

\author{Neil Barton}
\address[Neil Barton]{IFIKK, Universitetet i Oslo\\ Postboks 1020, Blindern \\
0315, Oslo\\
Norway}
\email{neil.barton@uni-konstanz.de}
\urladdr{https://neilbarton.net/} 

\author{Kameryn J. Williams}
\address[Kameryn J. Williams]{
Sam Houston State University \\
Department of Mathematics and Statistics \\
Box 2206 \\
Huntsville, TX 77341-2206 \\
USA}
\email{kameryn.j.w@shsu.edu}
\urladdr{http://kamerynjw.net}

\thanks{Order of authors is alphabetical. We would like to thank Carolin Antos, Julian Bacharach, Laura Crosilla, John Baldwin, Salvatore Florio, \O{}ystein Linnebo, Beau Madison Mount, Naomi Osorio-Kupferblum, Martin Pleitz,  Sam Roberts, and audiences in Konstanz and New York for comments and discussion. The first author is very grateful for the support of the VolkswagenStiftung via the project \textit{Forcing: Conceptual Change in the Foundations of Mathematics} and Norges forskningsr\r{a}det (the Research Council of Norway) for their support via the project \textit{Infinity and Intensionality: Towards A New Synthesis} (no. 314435).}

\date{29 July 2021}

\begin{document}

\maketitle

\begin{abstract}
 We explain and explore class-theoretic potentialism---the view that one can always individuate more classes over a set-theoretic universe. We examine some motivations for class-theoretic potentialism, before proving some results concerning the relevant potentialist systems (in particular exhibiting failures of the $\mathsf{.2}$ and $\mathsf{.3}$ axioms). We then discuss the significance of these results for the different kinds of class-theoretic potentialist.
\end{abstract}

\section*{Introduction}

In this paper we examine a new kind of \textit{potentialism} in set theory. From the off, let's state the difference between (set-theoretic) \textit{actualism} and \textit{potentialism}:

\begin{description}
\item[Set-Theoretic Actualism] There is a maximal completed universe of sets.
\end{description}

One natural such position is universist set-theoretic actualism; the view that there is \textit{exactly one} such universe. However, this is not necessary for actualism; one could have multiple distinct incomparable universes, each of which cannot be extended.\footnote{For the arguments of this paper, nothing much hangs on the matter.} Whatever one's preference, actualism contrasts sharply with:

\begin{description}
\item[Set-Theoretic Potentialism] The universe of sets is not a completed totality, but rather unfolds gradually as parts either come into existence or become accessible to us.\footnote{This statement closely mirrors \cite[p. 1.]{HamkinsLinnebo2018a}}
\end{description}

A common way of making set-theoretic potentialism mathematically precise is by viewing this `gradual unfolding' as describing a space (or spaces) of possible worlds. Modal operators are then often introduced, and we are able to ask several kinds of questions, including (i) what holds non-modally at particular worlds, (ii)  what modal principles certain worlds satisfy, and (iii) what the modal logic of different accessibility relations are. For this reason, discussion of potentialism often focuses on the nature of these set-theoretic worlds. For example, we might (inspired by \cite{Zermelo1930a}) view the worlds as the study of ever larger $V_\kappa$ for $\kappa$ inaccessible, with accessibility being coextensive with the subset relation. Another species of set-theoretic potentialism is to have the worlds be all those that can be obtained by set forcing (and moving to ground models) from some starting universe, and have one world $V$ be accessible from another $V'$ just in case one can force from $V'$ to obtain $V$. There have been several results in this field, including isolating the modal logic of forcing \cite{HamkinsLoewe2008a} and the study of potentialist maximality principles \cite{Hamkins2003a}, \cite{HamkinsLinnebo2018a}. One key question (that we will deal with in detail later in this article) concerns whether the modal axioms $\mathsf{.2}$ and $\mathsf{.3}$ are satisfied in addition to $\Sfour$ in the relevant potentialist systems.\footnote{It should be noted that $\mathsf{S4}$ trivially holds in any potentialist system. See \S\ref{modallogics}, which also has definitions of all these modal axioms.} 
$\mathsf{.3}$ indicates a kind of `inevitability' or `linearity' to how the worlds unfold, and $\mathsf{.2}$ indicates a form of `convergence' present on the frame. Moreover mirroring theorems (which allow us to move between potentialist and non-potentialist theories via a natural translation) are known to hold on systems containing $\mathsf{S4.2}$. This has lead some authors (e.g. \cite[p. 33]{Hamkins2018a}) to claim that the convergent forms of potentialism (i.e. those with modal logic at least $\mathsf{S4.2}$) are `implicitly actualist'. Whatever one thinks of these specific claims, it is clear that $\mathsf{.2}$ and $\mathsf{.3}$ represent clear dividing lines between different potentialist systems.

This greater understanding of set-theoretic potentialism has occurred alongside an explosion in the study of \emph{class theory} (also called \emph{second-order set theory}).
These theories introduce a new kind of variable to range over classes as well as sets, which are then governed by theories such as $\GB$, $\MK$, and their variants. Previously the key debate to be settled was whether or not the comprehension axiom for classes should be fully \textit{impredicative} or rather whether only \textit{predicative} comprehension was licensed by our conception of classes (see, for example, \cite{Uzquiano2003a}). However, recent literature has investigated an entire space of possible class theories with varying consequences and consistency strengths, analagous to the study of subsystems of second-order arithmetic. The dimension in which they can vary is not just along the axis of strength of comprehension. For example, class-theoretic versions of the axiom of choice are independent even of full impredicative $\MK$ (see \cite{GHK2019}). %Moreover, \textit{class} forcing (where the forcing partial order and generic can be proper classes) has provided us with a controlled yet flexible method for adding \textit{classes} to a model, in a similar way to how set forcing does for \textit{sets}.

These developments suggest that the various kinds of set-theoretic potentialism are not the whole story. Instead, we might study the following kind of potentialism:

\begin{description}
\item[Class-Theoretic Potentialism] The classes of the universe do not constitute a completed totality, but rather unfold gradually as more classes either come into existence or become accessible to us.
\end{description}

One need not be a set-theoretic potentialist if one is a class-theoretic potentialist. One can perfectly well have the classes over a model change whilst the sets remain fixed, say if one were a set-theoretic actualist. In this paper, we'll look at varieties of class-theoretic potentialism under both set-theoretic actualism and set-theoretic potentialism. 
We will argue for the following claims:

\begin{enumerate}
\item Class-theoretic potentialism can be motivated on the basis of several different philosophical conceptions of classes.
\item Class potentialist systems can exhibit a variety of modal structures, with some validating $\Sfourthree$ and $\Sfourtwo$, while others will fail to validate $\mathsf{.2}$ or $\mathsf{.3}$.
\item The various motivations for class-theoretic potentialism suggest different ways of modelling the position and distinct pictures of the modal behaviour of classes.
\end{enumerate}

%The strategy of our argument is to show how class-theoretic potentialism relates to different conceptions of classes, mathematically articulate the position and prove some results, and then discuss the relevant philosophical implications. Here's the plan in more detail:

Here's the plan: In \S\ref{motivations} we'll outline some philosophical positions regarding classes that can be used to motivate class-theoretic potentialism. We'll divide these into two broad kinds: Bottom-up approaches start with some \textit{fixed stock} of classes and then \textit{individuate} new classes over these, whereas top-down approaches see class-theoretic potentialism as arising from \textit{referential indeterminacy} and the ways we can \textit{interrelate} sharpenings of the ranges of the class-theoretic variables. \S\S\ref{systems}--\ref{modallogics} set up the key mathematical notions we shall use to examine these views, namely \textit{potentialist systems} (structures that formalise the notion of worlds and accessibility between them) and the \textit{modal logics} and \textit{axioms} they satisfy. \S\ref{truth-potentialism} proves some results about some potentialist systems, in particular showing that for weak theories below the level of $\GB + \ETR$, given suitable assumptions we can exhibit failures of the $\mathsf{.3}$ (Theorem \ref{thm:s42}) and $\mathsf{.2}$ axioms (Theorem \ref{thm:killing-truth}), showing that some systems have non-inevitability and radical branching. %(For the less technically-inclined reader,  skimming this section on a first read shouldn't unduly impair one's ability to understand the later sections.)
We'll then discuss some implications for bottom-up approaches to classes (\S\ref{bottom-up}), arguing that whilst there are good motivations for handling truth predicates, global choice is problematic in this context, and care is required in setting up the `right' model theory. We then (\S\ref{top-down}) examine top-down approaches, arguing that our results are indicative of more natural cases of radical branching than is normally seen in many potentialist contexts.
\S\ref{conclusion} provides some concluding remarks and identifies several open questions and directions for future research.

\section{Motivating Class-Theoretic Potentialism}\label{motivations}

A mathematically popular and expressively parsimonious conception of classes is to regard talk of classes as merely shorthand for certain formulas holding within the universe. So, for instance, ``$x \in \Ord$'' can be rendered as ``$x$ is a transitive set linearly ordered by $\in$'', or for a more complicated class such as a proper-class-sized embedding $j$, the formula ``$j(x) = y$'' can be rendered as some formula $\phi(x,y)$ only referring to sets, possibly involving parameters. A disadvantage of this view is that it appears to trivialise various mathematical theorems under their natural interpretation, such as \cite{Kunen1971a}'s result that there is no nontrivial $j: V \to V$ and the work by  \cite{VickersWelch2001a} on embeddings from inner models to the universe.\footnote{Briefly: if a nontrivial $j : V \to V$ is definable without parameters, then so is its critical point $\crit j$, the least ordinal moved by $j$. But elementary embeddings must fix all sets definable without parameters, so we would get that $j(\crit j) = \crit j$, contrary to the definition of critical point. A small extra argument gets that we further cannot have $j$ definable with (set) parameters, by looking at such $j$ with minimal critical point (for details see \cite{Suzuki1999a}). A similar argument works for consideration of $j: M \to V$ (as in \cite{VickersWelch2001a}).

The point that Kunen's Theorem appears to be trivialised when all embeddings are viewed as definable was communicated to us by
%[name removed for blind review]
Sam Roberts 
and is raised in \cite{HamkinsKirmayerPerlmutter2012a}, \cite{Fujimoto2019a} (esp. \S{}III), and \cite{Barton2021f} (esp. \S{}2.2). Of course there is much to say here, it is not as though the definibalist about classes \textit{cannot} give a nontrivial interpretation of these theorems (for example Kunen's Theorem can be interpreted as showing that there is no $j: V_{\lambda +2} \to V_{\lambda +2}$) but the point remains that a non-definabilist has a nice interpretation of the theorem that the definabilist lacks.}   
Another problem point is that some recent work in class theory, such as work on the strength of the class forcing theorem,\footnote{See \cite{HolyKrapfLuckeNjegomirSchlicht2016a}, \cite{HolyKrapfSchlicht2018a}, and \cite{GHHSW2020}.}
makes necessary use of undefinable classes such as truth predicates.
For this reason, we'll set this interpretation to one side. 

Instead, in this section we'll outline the links between class-theoretic potentialism and a variety of conceptions of classes. Our aim is not to conclusively argue for class-theoretic potentialism or exhaustively survey the views that might support it. Rather we aim to show that it is a viable position (philosophically speaking) and fits nicely with a several different positions, and indeed is even tenable if one is a set-theoretic actualist. We'll divide these into \textit{bottom-up} and \textit{top-down} approaches. The former roughly correspond to  processes of \textit{individuation} and the latter to considerations regarding \textit{indeterminacy of reference}. At this stage, we'll keep things relatively informal, before discussing the details of each proposal once the mathematical machinery is in place.

%Suppose, for the moment, that one has accepted 
%\sloppy set-theoretic actualism---that is,
%one thinks there is a completed universe of \textit{sets} which cannot be added to. %For ease of expression, we'll speak as if there's just one such universe, but nothing we say will change if there are multiple such. It still behooves the set-theoretic actualist to explain what \textit{classes} are. 

%We'll now survey some of the options that have been proposed, and explain how they might motivate class-theoretic potentialism. By reviewing the literature we can extract two main strands: \textit{bottom-up} approaches are those view classes as given to us via some iterated process of \textit{individuation}, and \textit{top-down} approaches are those that view classes as existing within a potentialist framework in virtue of \textit{indeterminateness of reference}.

\subsection{Bottom-up approaches}

The key feature of a bottom-up approach to classes is that one begins with some antecedently specified classes (e.g. the definable ones) and then builds up the classes by forming new classes from old. 

\subsubsection{Liberal Predicativism.} The first interpretation we'll look at is derived from the work of Parsons (e.g. \cite{Parsons1974a}) with subsequent development by Fujimoto (e.g. \cite{Fujimoto2019a}) and views classes as \textit{predicate extensions}:

\begin{description}
\item[Class Predicativism] Classes are extensions of predicates.\footnote{We should note that this liberalised sense of predicativism contrasts somewhat with the usual sense of the term as developed historically by Poincar\'{e}, Russell, and Weyl. Thanks to 
Laura Crosilla 
%[name removed for blind review]
for pointing out this issue.}
%\footnote{%Fujimoto appears to suggest in \cite[p. 209]{Fujimoto2019a} that classes just \textit{are} predicates whereas Parsons characterises them as predicate extensions  (see \cite[p. 7]{Parsons1974a}]). Either way, the point remains the same: Predicates are parts of language, rather than the combinatorially characterised objects of iterative set theory. 
%We should also note that this liberalised sense of predicativism contrasts somewhat with the usual sense of the term as developed historically by Poincar\'{e}, Russell, and Weyl. Thanks to 
%Laura Crosilla 
%[name removed for blind review]
%for pointing out this issue.}
\end{description}

Once this perspective has been taken, class-theoretic potentialism becomes a natural position. Simply put: Because classes are given by \textit{language} we might think that for any language there is another that nontrivially extends it. In particular, if we think that we can always extend by a truth predicate for any given language, we will obtain an infinite hierarchy of languages. This is borne out in the way that both Parsons and Fujimoto express their position. Take the following illustrative passage from Parsons:

\begin{quote} 
...we do not have an independent understanding of what predicates or abstracts denote, or what class or second-order variables range over. It follows that ``all extensions...'' will, unless set-theoretic notions are imported, only mean ``the extensions of all possible predicates''. And it seems evident that the ``totality'' of possible predicates is irremediably potential...
\cite[p. 8]{Parsons1974a}
\end{quote}

%\begin{quote}
%Our proposal is to interpret the quantifier $\exists X$ as ``there exists an admissible predicate such that...'' or ``there is a predicate we may admissibly introduce such that...'' and interpret the membership relation $x \in X$ as ``the predicate $X$ holds for $x$.'' \cite[p. 212]{Fujimoto2019a}
%\end{quote}

Similar sentiments are available in Fujimoto, who argues that predicates can be ``admissibly introduced'' and agrees with Parsons that there is no determinate totality of admissible predicates \cite[pp. 212, 214--215]{Fujimoto2019a}. 
 
Both Parsons' and Fujimoto's words suggest potentialist readings. However, they each opt for non-modal theories. Parsons uses a theory of classes and satisfaction and Fujimoto adopts $\GB$ augmented with a class-theoretic principle satisfying a version of the $\mathsf{KF}$ truth-theoretic axioms.\footnote{See \cite[p. 223]{Fujimoto2019a} for details.}

Given that both Parsons and Fujimoto think there is no definite collection of all predicates, it is natural to consider a modal class-theoretic potentialist framework in ascertaining the prospects for class predicativism. Beginning with some fixed language, we come to individuate new classes by adding predicates for them into our language. The classes for them are thus ``irremediably potential'', much like the stock of predicates we may admissibly introduce.

\subsubsection{Property Potentialism\footnote{We are grateful to 
%[name removed for blind review]
\O{}ystein Linnebo
for suggesting this interpretation of classes as leading to class-theoretic potentialism and Sam Roberts for several helpful discussions here.}}

A somewhat similar approach is to view classes as arising from {\it property application} relation.
%\footnote{One might think that predicativism and property-theoretic accounts are closely linked, since properties are often taken to be the semantic values of predicates, with an application of a property to an object corresponding to a predicate holding of that object (or a name for that object). We don't need to take a stand on this issue and so relegate this comment to a footnote, but it bears mentioning that there may be relations between the different positions on classes.} 
This kind of view, when taken in unrestricted generality, leads quickly to semantic paradoxes 
 via the consideration of a property $P$ that applies to all and only the non-self-applicable things ($P$ is then self-applicable iff it is not self-applicable, exactly as in Russell's Paradox or the paradox of heterologicality). A restriction is therefore needed, and one suggestion, made by both \cite{Linnebo2006a} and \cite{Fine2005b}, is that the application relation for properties is successively individuated.
Linnebo individuates properties along the ordinals, and thereby uses a length $\Ord$ iteration, while Fine considers longer iterations. This successive individuation of the application relation can then lead to a version of class-theoretic potentialism by obtaining different domains of classes by letting each domain be those classes that are co-extensional with properties at some stage of the iteration.\footnote{Whilst \cite{Fine2005b} does not expressly use property-theoretic language he uses the same picture of successively specifying an application relation, and so we include them together. Linnebo's theory and Fine's theory are also similar from a mathematical perspective. The strength of these theories has recently been examined by Roberts (in \cite{RobertsProperties}). These property theories are provably consistent in $\GB$ +  $\Pi^1_1$-Comprehension. More strength can be obtained by the addition of a reflection principle for properties in this context.
%In particular, the principle that (for $\phi$ in the language of property theory) ``$\phi$ holds concerning all properties if and only if $\phi$ true at some stage of the iteration'' yields full impredicative comprehension when we let classes be interpreted in the above manner. 
Moreover, some of these property theories are mutually interpretable with $\ZFC$ with the addition for $\Ord$-iterated truth predicates (see \cite{RobertsTruth}).}

Whilst both Fine and Linnebo use a non-modal framework for properties (within which we can interpret certain fragments of class theory) their philosophical claims admit a potentialist interpretation. Talk of successively individuating the application relation (and the new classes that are available each time we do) can be thought of in a potentialist manner. Indeed, one can see this potentialism concerning the classes that are the extensions of properties in the sets as being given formal codification in a stage-theoretic version of Linnebo's property theory given by \cite{RobertsProperties}. Philosophically speaking though, the way that the two constructions are phrased also suggest this interpretation. For example, Linnebo writes:

\begin{quote}
We begin by individuating some class of set-theoretic properties. For concreteness, assume we individuate those set-theoretic properties definable...allowing for parameters referring to pure sets. Now we want to use the set-theoretic properties we have just individuated to individuate more properties. 
\cite[p. 173]{Linnebo2006a}
\end{quote}

Fine expresses a similar sentiment, in the context of contrasting his approach with the traditional cumulative hierarchy picture:

\begin{quote}
%On the usual conception of the cumulative hierarchy of Zermelo-Fraenkel set theory ($\mathsf{ZF}$), we think of the membership predicate as given and of the ontology of sets or classes as something to be made out. Thus given an understanding of membership, we successively carve out the ontology of sets by using the membership predicate to specify which further sets should be added to those that are already taken to exist. Under the present approach, by contrast, we think of the ontology of classes as given and of the membership predicate as something to be made out. Thus

...Given an understanding of the ontology of classes, we successively carve out extensions of the membership predicate by using conditions on the domain of classes to specify which further membership relationships should obtain. \cite[p. 547]{Fine2005b}
\end{quote}

Given this idea of successively individuating the application relation, there will be new classes appearing as the application relation  progressively individuates certain properties as holding of more and more sets. For instance as we move through the first few levels of individuation, we will individuate a truth-predicate for the language of $\ZFC$ (after the first stage), and then a truth predicate for the expanded language, and so on.%\footnote{For details, see \cite{RobertsTruth}.}
\smallskip

That concludes our introduction of bottom-up approaches to class-theoretic potentialism. There may be others, but this is not so important for our purposes---we just want to motivate consideration of the view, not provide a comprehensive description of its possible motivations.

To review, the key facets of bottom-up approaches are:

\begin{description}
\item[Initial World] We obtain classes \textit{beginning} with some \textit{initially} specified classes and then...
\item[Individuation] ... we \textit{individuate} new classes over the existing classes.
\end{description}

We will save some further discussion of \textbf{Initial World} and \textbf{Individuation} for when we have some more mathematical tools in play. For now, let's move on to \textit{top-down} approaches.

\subsection{Top-down approaches}

%A different route to class-theoretic potentialism is \textit{top-down} in nature.
Instead of \textit{starting} with some given collection of classes and iterating a process of individuation, we might instead view potentialism as arising out of \textit{referential indeterminacy}. This is the core approach of \textit{top-down} views: We state some conditions we would like domains of classes to satisfy, but it may be that there is no single domain that is thereby referred to. We can then take class-theoretic potentialism to be telling us how we may move around within these domains that satisfy our basic class-theoretic principles. More concretely, we may see the following views as motivating class-theoretic potentialism.

\subsubsection{Multiverse approaches to class-theoretic potentialism}

The first is relatively simple in nature---we may view class-theoretic potentialism as being motivated by garden-variety set-theoretic potentialism. If one thinks that any universe of set theory appears as a set in a larger universe (i.e. for universe $V$ there is another universe $V'$ such that $V \in V'$) \textit{and} that any universe can be extended by set forcing,\footnote{Views of this kind include \cite{Hamkins2012a}, \cite{ArrigoniFriedman2013a}, and \cite{Scambler2021a}.} then class-theoretic potentialism considers multiversally-interesting set-theoretic structures. For example given a universe $V$, we can always make $V$ countable by moving to a universe $V'$ in which $V$ appears as a set, and then forcing over $V'$  to add a bijection $V \to \omega$ (call this universe $V'[G]$), via the usual collapse forcing to make a set countable. Within $V'[G]$, we can consider class-theoretic potentialist systems consisting of $V$ equipped with different possible collections of classes (to be discussed in greater detail in \S\ref{systems}).

If one thinks that the concept of \textit{arbitrary set} is indeterminate (and holds a multiversism on these grounds), one is likely to hold also that our concept of \textit{class} is also indeterminate. Thus, even if the set-multiversist fixes some universe $V$ as a starting position (within the multiverse), it is unlikely to be determinate exactly what classes exist over $V$. We can then construe reference to all classes of $V$ as referentially indeterminate and yielding a potentialist system of classes.%\footnote{This bears a similarity to how the usual set forcing potentialist may think that the natural numbers are a fixed totality but hold that the reference to \textit{all sets of natural numbers} is determinate.}

\subsubsection{Plurals and Potentialism}

Plural resources have been used to interpret proper class talk (see, for example, \cite{Uzquiano2003a}).
Often the ranges of plural variables are taken to be determinate (e.g. \cite{Hossack2000a}, \cite{Uzquiano2003a}). However, this view has recently been challenged by the work of Florio and Linnebo 
\sloppy(in \cite{FlorioLinnebo2016a})
who show that there are versions of Henkin semantics for plural logic, and argue that this calls into question the determinacy of plural quantification. 

It is out of the scope of this paper to settle whether or not one should accept that plural resources are in fact indeterminate. Our focus here is on examining a range of possible views that might motivate the more general idea of class-theoretic potentialism, rather than trying to settle this tricky matter. However, \textit{if} one does accept that such resources are indeterminate, one might motivate a class-theoretic potentialism via the plural interpretation of classes. If we hold that there is referential indeterminacy in the ranges of the plural variables, and we're interpreting classes as pluralities over our fixed universe, then the referential indeterminacy concerning plurals transfers immediately to referential indeterminacy about classes. Once we have this indeterminacy in the picture, it is a short step to class-theoretic potentialism, understood as the study of different \textit{precise} interpretations of the plural variables and how we may move between these interpretations. Much as a \textit{set-theoretic} multiversist can extract a \textit{set-theoretic} potentialism from their multiversism,\footnote{See the introduction of \cite{HamkinsLinnebo2018a} for a brief discussion of the move from set-theoretic multiversism to set-theoretic potentialism.}
so a theorist who is effectively a multiversist about the \textit{pluralities} that exist (and hence a multiversist about \textit{classes} too) can extract a \textit{class-theoretic} potentialism from their kind of multiversism.

So, to conclude this section, in addition to bottom-up approaches and their twin pillars of \textbf{Initial World} and \textbf{Individuation}, we have top-down approaches based on the following two ideas:

\begin{description}
\item[Referential Indeterminacy] Over a given universe of sets $V$, reference to \textit{the classes of $V$} is not determinate (i.e. does not pick out a unique privileged interpretation).
\item[Interrelation of Interpretations] Class-theoretic potentialism can be understood as interrelating these distinct possible interpretations (e.g. how one can move between them, what theories they satisfy, etc.).
\end{description}

\section{Class-theoretic potentialist systems}\label{systems}

With these motivations for the broad idea of class-theoretic potentialism in hand, it is time to lend some mathematical precision to their study. In this section we'll discuss some class-theoretic potentialist systems and their relation to bottom-up and top-down approaches.

\subsection{Class-theoretic principles}

We'll use a two-sorted approach to class theory, with \emph{sets} and \emph{classes} as the two types of objects.\footnote{For the reader who wants a more thorough exposition of class theories than we give, we recommend \cite[\S2]{williams2019} for a modern treatment.} 
A model of class theory will be denoted $(M,\Xcal)$, where $M$ is the sets and $\Xcal$ is the classes. Our focus will be transitive models, for whom their membership relation is the true $\in$, and we will suppress the membership relation in the notation.\footnote{The reason for this restriction is most researchers in this area view nonstandard models as unintended, ersatz models. For the reader who does not share this view, we remark that some of our technical results about countable transitive models can be reworked to apply to nonstandard models. This requires more care and complicates the exposition, so we leave it to such a reader to check what generalises to that context.}

Call a formula in the language of class theory \emph{elementary} if its quantifiers only occur over set variables (but class parameters are allowed).

\begin{definition}[Class theories]
All our class theories will include $\ZFC$ for the sets. Where they differ is in their axioms for classes.\footnote{Though these class axioms can have consequences for the sets, e.g. $\Con(\ZFC)$.} 
They also include an extensionality axiom for classes and a replacement axiom for classes---if $F$ is a class function and $a$ is a set then $F''a$ is a set.
\begin{itemize}
\item Adding the predicative comprehension schema, viz. the instances of comprehension for elementary formulas, gives \emph{G\"odel--Bernays class theory} $\GB$.
\item Adding the full impredicative comprehension schema, viz. all instances of comprehension, including those with class quantifiers, gives \emph{Morse--Kelley class theory} $\MK$.
\end{itemize}
\end{definition}

Beyond these two class theories certain class theoretic principles will arise in our investigation.

\begin{definition}
\emph{Global choice} is the assertion that there is a global choice function for all nonempty sets. Equivalently, it may be formulated as the assertion of a global well-order. Let $\GBC$ denote $\GB$ + global choice.%\footnote{Given a global well-order, you get a global choice function by picking the least element according to the order. Given a global choice function, pick a well-order for each $V_\alpha$, which exist by $\AC$ for sets. You can then combine these set-sized well-orders into a global well-order by ordering $x$ and $y$ first by their rank and, if they have the same rank $\alpha$ according to the well-order you chose for $V_\alpha$.} %% kill footnote? Should be obvious to any reader. --KJW
\end{definition}

It is well-known that $\MK$ does not imply global choice, but that global choice has no consequences for sets. Given a model of class theory a generic global well-order can be added by a forcing which does not add sets. See \cite{Felgner1971} for a classic treatment of this construction.

\begin{definition}
\emph{Elementary transfinite recursion} $\ETR$ is the principle asserting that transfinite recursion of elementary properties along well-founded classes have solutions. (See \cite[\S4]{GitmanHamkins2017a} for a more precise formulation.)
\end{definition}

Observe that $\MK$ proves $\ETR$, since one can define the solution to an elementary transfinite recursion  with an impredicative formula. On the other hand, $\ETR$ exceeds $\GB$ in consistency strength, since the Tarskian truth predicate for $V$ can be given by elementary transfinite recursion, and thereby $\ETR$ proves the consistency of $\ZFC$.

Indeed, $\ETR$ is closely connected to truth predicates, and can equivalently be expressed as a truth-theoretic principle.

\begin{theorem}[Fujimoto \cite{fujimoto2012}]
Over $\GB$, $\ETR$ is equivalent to the assertion that for any class $A$ and any class well-order $\Gamma$ the length $\Gamma$ iterated truth predicate relative to the parameter $A$ exists.\footnote{See the beginning of \S\ref{truth-potentialism} for definitions and fuller discussion of truth predicates and iterated truth predicates.}
\end{theorem}

We caution the reader that for this theorem to be true it must include class well-orders far longer than $\Ord$. It's straightforward to define a well-order of ordertype $\Ord + 1$, or $\Ord + \Ord$, or $\Ord^2$, following similar constructions on $\omega$. Stronger class theoretic principles allow one to define longer and longer well-orders, akin to how stronger subsystems of second-order arithmetic allow the definition of longer and longer countable well-orders.

One can restrict $\ETR$ to get a hierarchy of transfinite recursion principles. If $\Gamma$ is a class well-order let $\ETR(\Gamma)$ denote the restriction of $\ETR$ to recursions along well-founded classes of rank $\mathord{\le}\Gamma$ and let $\ETR(<\Gamma)$ denote the restriction of $\ETR$ to rank $\mathord{<}\Gamma$. These principles separate from full $\ETR$ and from each other based on $\Gamma$, according to consistency strength; see \cite{williams2019} for details.

The existence of models of any of these class theories is beyond $\ZF$ in consistency strength. But the large cardinal strength is very mild. All constructions we study are below the level of one inaccessible cardinal.

\subsection{Potentialist systems}

A \emph{potentialist system} is a collection of structures of the same type, ordered by a reflexive and transitive relation $\subseteq$ which refines the substructure relation. Potentialist systems give a formalisation of a domain we can think of as dynamically growing. In this section, we will provide the basic definitions for the potentialist systems we plan on considering. Our two philosophical approaches to class-theoretic potentialism (viz. bottom-up and top-down) will correspond to two different ways of studying potentialist systems mathematically. In both cases, we will consider the sets as a fixed totality, with the potential nature entirely in the classes.

For bottom-up approaches, the potentialist systems will be specified by rules which specify the base world and how worlds may be extended. We will look at extension rules based on adding truth predicates and certain class forcing extensions; we leave other extension rules to future work. We then ask which if any potentialist systems satisfy these rules. For example, we will show that the rules corresponding to liberal predicativism and property potentialism admit smallest systems satisfying those rules.

For top-down approaches, we consider all possible collections of classes which meet some basic criteria. For $M$ a model of $\ZFC$, say that a \emph{top-down class potentialist system for $M$} is the collection of all worlds $(M,\Xcal) \models T$ satisfying a fixed property, where $T$ is a class theory. For ease of writing we call such $\Xcal$ a \emph{$T$-expansion for $M$}. We will consider two choices for $M$ and this fixed property, corresponding to different philosophical commitments. Corresponding to set multiversism we will look at countable transitive $M$ and the top-down class potentialist system for $M$ consisting of all countable $(M,\Xcal) \models T$. Additionally, we will make extra assumptions on $M$ to ensure there actually are $T$-expansions for $M$; cf. Proposition~\ref{prop:worldly}. %% say why no nonstandard models? say why this restriction?
Corresponding to set universism we will look at $M = V_\kappa$ for some inaccessible $\kappa$ and the top-down class potentialist system for $M$ consisting of all $(M,\Xcal) \models T$ with $\card \Xcal = \kappa$. This restriction avoids the trivial case of allowing the world $(M,\powerset(M))$, which would be morally nothing more than class theoretic actualism.

%It could be that $M$ has no $T$-expansion and . For example, $\MK$ has first-order consequences which go beyond $\ZFC$ and so $M$ might have no $\MK$-expansions due to having a bad theory, e.g. if $M$ has no worldly cardinals. We will implicitly assume that this does not happen, and that our potentialist system is nontrivial.

For both approaches, an important goal will be to understand the modal logic of the potentialist systems in play. We'll provide some definitions and background for this before we give a more fine-grained analysis.

\section{Modal logics of class-theoretic potentialism}\label{modallogics}

Given a potentialist system $(\Afrak, \subseteq)$ there is a natural modal interpretation: $\necessary \phi$ holds at a world $M \in \Afrak$ if $\phi$ holds in every $N \supseteq M$ with $N \in \Afrak$ and $\possible \phi$ holds at a world $M \in \Afrak$ if $\phi$ holds in some $N \supseteq M$ with $N \in \Afrak$. 
We will look only at modal formulae which are obtained from propositional modal formulae by substituting class theoretic formulae for the propositional variables. Quantifiers are allowed in the class theoretic formulae we substitute, but we do not work with quantified modal logic per se.
A propositional modal formula is \emph{valid} at $M \in \Afrak$ if it's true for any substitution of class theoretic formulae for the propositional variables. 
%When talking about which modal formulae are valid, we will restrict our attention to those without quantifiers; however, we will allow quantifiers in the class theoretic sentences we substitute for the propositional variables.
For example, $\necessary p \impl p$ is valid at every world because if $\phi$ holds in every extension of $M$ it in particular holds at $M$.\footnote{For the reader unfamiliar with modal logic looking for a general resource we recommend \cite{cocchiarella-freund2008} and \cite{BdRV2002}. For the application of modal logic to set theoretic potentialism we recommend \cite{HamkinsLoewe2008a} and \cite{HamkinsLinnebo2018a}.}

In the particular case of class-theoretic potentialist systems, the extension relation $(M,\Xcal) \subseteq (M,\Ycal)$ is purely in the classes---the sets are fixed. So $(M,\Xcal) \models \possible \phi$ if there is a larger system of classes which satisfies $\phi$ and $(M,\Xcal) \models \necessary \phi$ if every larger system of classes satisfies $\phi$. For these systems, it will not be possible to express $\necessary \phi$ nor $\possible \phi$ in the language of class theory. (Contrast with the modal logic of forcing \cite{HamkinsLoewe2008a} where the modal operations are expressible in the language of set theory.)

A key mathematical question is then: what are the \emph{modal validities} of $\Afrak$, the collection of modal assertions valid at every world in $\Afrak$? And one can refine this question, by asking whether this is allowing or disallowing parameters in formulas changes the validities, or by asking whether it varies by world.

\begin{definition}
$\Sfour$ is the propositional modal theory obtained by closing under the inference rules of \emph{modus ponens} and necessitation (viz. if $\phi$ is a theorem then so is $\necessary \phi$), starting with the following axioms, for any propositional modal formulae $\phi$ and $\psi$:
\begin{align*}
(\mathsf{K})\qquad & \necessary (\phi \impl \psi) \impl (\necessary \phi \impl \necessary \psi) \\
(\mathsf{Dual})\qquad & \neg \possible \phi \iff \necessary \neg \psi \\
(\mathsf{T})\qquad & \necessary \phi \impl \phi \\
(\mathsf{4})\qquad & \necessary \phi \impl \necessary \necessary \phi
\end{align*}
\end{definition}

$\mathsf K$ and $\mathsf{Dual}$ hold for any modal logic which comes from a Kripke frame; they come from the corresponding rules for universal quantification applied to quantifying over worlds in the frame. $\mathsf T$ holds on any frame whose accessibility relation is reflexive, and $\mathsf 4$ holds on any frame whose accessibility relation is transitive. Potentialist systems are by definition reflexive and transitive, so they always validate $\Sfour$.

Some other modal axioms we consider in this article are:

\begin{definition}
\begin{align*}
(\mathsf{.2})\qquad & \possible \necessary \phi \impl \necessary \possible \phi \\
(\mathsf{.3})\qquad & (\possible \phi \land \possible \psi) \impl \possible [(\phi \land \possible \psi) \lor (\possible \phi \land \psi)]
\end{align*}
Adding these to $\Sfour$ gives, respectively, the theories $\Sfourtwo$ and $\Sfourthree$.
\end{definition}

It is straightforward to check that $\mathsf{.2}$ is valid for any frame whose accessibility relation is directed and $\mathsf{.3}$ holds for any frame whose accessibility relation is linear. The converses, however, do not hold. Counterexamples to both can be found in the modal logic of forcing. The modal logic of forcing allowing any poset is $\Sfourtwo$ \cite{HamkinsLoewe2008a}, but by an old result of Mostowski \cite{Mostowski1976a} there are forcing extensions with no common width-extension to a model of $\ZFC$. And the modal logic of collapse forcings is $\Sfourthree$ \cite{hamkins-leibman-lowe2015}, but this potentialist system is not linearly ordered since you could pick different generics for the same collapse forcing. Nonetheless, these frame conditions are a fruitful heuristic for thinking about these modal axioms.

We wish to briefly discuss this Mostowski result, which was originally formulated in the context of class forcing with models of $\GB$. This result will imply that many of class potentialist systems we consider are not directed, and its proof foreshadows the argument for Theorem~\ref{thm:killing-truth}. We state a special case of his more general result; see the original paper \cite{Mostowski1976a} or the recent \cite{HHKVW2019} for further extensions.

\begin{theorem}[Mostowski]
Let $(M,\Xcal) \models \GB$ be countable and transitive. Then there are $C,B \subseteq \Add(\Ord,1)^M$ generic over $(M,\Xcal)$ so that from $C$ and $B$ can be defined over $M$ a real which codes an order of ordertype $\Ord^M$. Consequently, no $\GB$-expansion for $M$ can contain both $C$ and $B$.
\end{theorem}

\begin{proof}[Proof Sketch]
Let $r \subseteq \omega$ be such a real. The strategy is to place the bits of $r$ in the Cohen generic $C$ so that given both generics we can recover the coding points. The strategy for this is to alternate between building up $C$ and $B$ to meet the next generic in the list. At stage $n$, extend the partial construction of $C$ to meet the next dense set, then place the $n$-th bit of $r$. Meanwhile, extend the partial construction of $B$ by adding $0$s to have the same length as what we've built up for $C$. Then alternate: further extend the partial construction of $B$ to meet the next dense set, while putting $0$s on the end of the approximation for $C$ to have the same length. Given the generics $C$ and $B$ built up in this way, you can recover the coding points by following the alternating blocks of $0$s. Thereby, any $\GB$-expansion for $M$ which contained both $B$ and $C$ would think that $M$ has countable cofinality, but that would contradict the class replacement axiom
\end{proof}

The top-down class potentialist system consisting of all countable $\GB$-expansions for countable transitive $M$ will fail to be directed, by this theorem. More generally, any class potentialist system over a countable transitive universe of sets which allows extensions by Cohen generic classes ordinals will also fail to be directed. Nevertheless, it is possible for such a potentialist system to still validate $\mathsf{.2}$ and exhibit directedness in its modal truths; see Theorem~\ref{thm:s42}. On the other hand, we will see that the top-down class potentialist systems of all countable $\GB$-expansions fails to validate $\mathsf{.2}$; see Theorem~\ref{thm:killing-truth}. Contrast this to the modal logic of (set) forcing, where the Mostowski theorem applies but the modal validities are nonetheless $\Sfourtwo$.

As we discuss elsewhere the validity of $\mathsf{.2}$ provides a dividing line for potentialist systems. To use the language of \cite[\S7]{hamkins2018:arithmetic-potentialism}, systems for which $\mathsf{.2}$ is invalid exhibit a ``radical branching'' whereas systems for which $\mathsf{.2}$ is valid exhibit ``directed convergence''. Additionally, the validity of $\mathsf{.2}$ enables mirroring theorems, as in \cite{Linnebo2013a} or \cite{HamkinsLinnebo2018a}.

Given the information that we get from knowing a potentialist system's modal validities, we want tools which allow us to calibrate them. Lower bounds are relatively easy to determine---one simply shows that one can cook up worlds of the required kind---but showing that \textit{no more} is satisfied (i.e. upper bounds) is substantially harder. Let's briefly describe the main tools, \emph{control statements}, used to compute upper bounds.

\begin{definition}[\cite{HamkinsLoewe2008a,hamkins-leibman-lowe2015}]
\ 
\begin{itemize}
\item A \emph{button} is an assertion $\beta$ so that $\possible \necessary \beta$ holds at every world. If $\necessary \beta$ holds at a world $M$, we say $\beta$ is \emph{pushed} for $M$, otherwise we say $\beta$ is \emph{unpushed}. The intuition is, you can push a button, making $\beta$ true forevermore, but once you push it you can never unpush it.

\item A \emph{switch} is an assertion $\sigma$ so that $\possible \sigma$ and $\possible \neg \sigma$ holds at every world. The intuition is, you can toggle the truth value of $\sigma$ freely back and forth.

\item A \emph{ratchet} is a finite sequence $\rho_0, \ldots, \rho_n$ of buttons so that pushing $\rho_i$ pushes $\rho_j$ for all $j<i$. The intuition is, you can ratchet forward but never back.

\item A \emph{long ratchet} of length a limit well-order $\Gamma$ is is a uniformly definable sequence of buttons $r_\xi$, indexed by $\xi < \Gamma$, so that pushing $r_\xi$ pushes $r_\eta$ for all $\eta < \xi$ and so that in no world are all buttons on the ratchet pushed.
\end{itemize}
A collection of control statements is called \emph{independent} if any subcollection of the control statement can be manipulated without affecting any of the other control statements.
\end{definition}

By showing that a potentialist system admits certain control statements, we get upper bounds for their modal validities. 

\begin{theorem}[\cite{HamkinsLoewe2008a}] \label{thm:buttons-and-switches}
If a potentialist system admits arbitrarily large finite families of independent buttons and switches then its modal validities are contained within $\Sfourtwo$.
\end{theorem}

\begin{theorem}[\cite{hamkins-leibman-lowe2015}] \label{thm:ratchets}
Suppose a potentialist system admits a long ratchet whose length $\Gamma$ is closed under addition of ordinals $\mathord{<}\omega^2$. Then, in any world where the ratchet is not fully pushed has its modal validities contained within $\Sfourthree$.\footnote{Hamkins, Leibman, and L\"owe formulated this result for long ratchets of length $\Ord$, but it is simple to check that what they used is that $\Ord$ is closed under addition of ordinals $\mathord{<}\omega^2$.}
\end{theorem}

With the setup of potentialist systems and their modal validities in hand, we can begin to examine class-theoretic potentialism mathematically and draw some philosophical conclusions on this basis. As we shall see, many potentialist systems violate the $\mathsf{.2}$ and $\mathsf{.3}$ axioms. These results will help us to raise some challenges for class-theoretic potentialism and help us to elucidate the position further, in particular relating these results to the philosophical motivations considered in \S\ref{motivations} (we do so in \S\ref{bottom-up} and \S\ref{top-down}).

\section{Truth predicates and potentialism}\label{truth-potentialism}

In this section we present some mathematical results about class potentialist systems which involve truth predicates.
Most of the results are phrased in terms of certain bottom-up potentialist systems, but some of the work also applies to top-down potentialist systems; failures of modal principles for the smaller systems can sometimes be pushed up to a larger system. The more philosophically-inclined reader can skim this section; in the subsequent sections we discuss the philosophical significance of the results in this section.

Let us begin by fixing some notation and definitions. We will use capital Greek letters to refer to class well-orders. Addition, multiplication, and exponentiation on these are defined as usual. To match the familiar notation for ordinals, we write $\xi < \Lambda$ to mean $\xi \in \dom(\Lambda)$. To compare elements $\xi$ and $\eta$ of $\Lambda$ we write $\xi < \eta$. Given class well-orders $\Lambda$ and $\Gamma$ say that $\Lambda$ is \emph{closed under addition $\mathord{<}\Gamma$} if whenever $\xi < \Lambda$ and $\eta < \Gamma$ we have that $\xi + \eta < \Gamma$. That is, $\Lambda$ is closed under addition $\mathord{<}\Gamma$ if every element of $\Lambda$ has an $\eta$-th successor in $\Lambda$ for each $\eta < \Gamma$.
Note that whether a class is well-founded can be expressed just by quantifying over sets. Accordingly any two $\GB$-expansions for $M$ agree on whether a common relation is well-founded. 

On the other hand, a model of $\GB$ may think that a class is well-founded while externally it is seen to actually be ill-founded. Say that $(M,\Xcal) \models \GB$ is a \emph{$\beta$-model} if it is correct about which of its classes are well-founded. A short argument yields that if $\Ord^M$ has uncountable cofinality then $(M,\Xcal)$ is a $\beta$-model. In contrast, if $\Ord^M$ has countable cofinality then $(M,\Xcal)$ may fail to be a $\beta$-model. Indeed, being a $\beta$-model is a strong property; for example, Marek and Mostowski showed \cite[Theorem 3.2]{marek-mostowski1975} that if $\beta$ is the least ordinal height of a $\beta$-model of $\MK$, then there are transitive models of $\MK$ of height $<\beta$.

Consider a fixed transitive $M \models \ZF$.
The \emph{truth predicate} for $M$ is the class $\Tr \subseteq M$ which satisfies the  recursive Tarskian truth conditionals for $(M,\in)$. In case we wish to emphasise for which structure $\Tr$ is a truth predicate we will write $\Tr^M$. Observe that the property of being $\Tr^M$ is uniquely determined by a formula which only quantifies over sets, and so any $\GB$-expansion of $M$ can verify whether a class is $\Tr^M$. Given a class $A \subseteq M$ the \emph{truth predicate relative to $A$} is the unique class $\Tr(A) \subseteq M$ which satisfies the recursive Tarski truth conditionals for $(M,\in,A)$, the expansion of $(M,\in)$ with a unary predicate symbol for $A$. Again, it is uniquely determined and can be recognised as such by any $\GB$-expansion of $M$.

Given the truth predicate $\Tr \subseteq M$, we can consider $\Tr(\Tr)$, the truth predicate relative to $\Tr$. This gives expanded resources from which more can be defined. For example, the theory $\GB$ $+$ ``$\Tr$ exists'' proves the consistency of $\ZFC$---restricting $\Tr$ to sentences gives a consistent completion of $\ZFC$. Since $\ZFC$ and $\GB$ are equiconsistent this theory proves the consistency of $\GB$. But this theory cannot prove its own consistency. However, $\GB$ $+$ ``$\Tr(\Tr)$ exists'' does prove the consistency of $\GB$ $+$ ``$\Tr$ exists'' by a similar argument. It won't prove its own consistency, but for that we could use $\Tr(\Tr(\Tr))$. We can continue to build this hierarchy of truth about truth about truth (and so on) higher, even into the transfinite. This hierarchy can be unified in the single definition of an iterated truth predicate. 

Working over our fixed transitive $M \models \ZF$, let $\Lambda$ be a well-order. A \emph{$\Lambda$-iterated truth predicate} is a class $\Theta$ of triples $(\xi,\phi,a)$ where $(\xi,\phi,a) \in \Theta$ intuitively means that $\phi(a)$ is true at level $\xi < \Lambda$.
Here $\phi$ is a formula in the language where we added a predicate symbol $\hat \Theta$ which is interpreted by $\Theta$. Formally, this is defined by a modified form of the Tarskian recursion, with an extra clause in the definition asserting that $(\xi,\godel{\hat \Theta(x,y,z)}, \seq{\eta,\phi,b}) \in \Theta$ if and only if $\eta < \xi$ and $(\eta,\phi,b) \in \Theta$. As with the ordinary case, we can consider iterated truth predicates relative to a class parameter $A$ via the obvious modification. See \cite[\S9]{GHHSW2020} for precise definitions.

$\GB$ suffices to prove the uniqueness of iterated truth predicates. Given two classes which satisfy the definition of being a $\Lambda$-iterated truth predicate relative to $A$, by predicative comprehension we can form the class of indices where they disagree. So if they disagree they must disagree at a minimal stage, from which we can derive a contradiction. We will use $\Tr_\Lambda$ to denote \emph{the} $\Lambda$-iterated truth predicate and $\Tr_\Lambda(A)$ to denote the $\Lambda$-iterated truth predicate relative to $A$. If $\xi < \Lambda$ then we write $\Tr_\xi$ to mean $\Tr_{\Lambda \rest \xi}$. 
Observe that, up to recoding, $\Tr_1$ is the same as $\Tr$ and $\Tr_{\Lambda + 1}$ is the same as $\Tr(\Tr_\Lambda)$, and similarly for relativised truth predicates. In general, $\Tr_\Gamma(\Tr_\Lambda)$ is inter-definable with $\Tr_{\Lambda + \Gamma}$

With these definitions in hand, let us now describe a species of class potentialist systems meant to capture the idea that we can always expand by adding truth predicates. First, a bit of notation. If $\Xcal$ is a collection of classes over $M$ and $A$ is a class over $M$ then let $\Xcal[A] \subseteq \powerset(M)$ be the smallest $\GB$-expansion for $M$ which extends $\Xcal$ and contains $A$.
%\footnote{The reader may worry about whether there is any such $\GB$-expansion for $M$. It is right to worry, since in general there need not be. But we will confine ourselves to a context in which it is defined.} %% kill footnote? --KJW
 Specifically,  $\Xcal[A]$ consists of the classes over $M$ definable using $A$ and finitely many classes from $\Xcal$.

\begin{definition} \label{def:truth-pot}
Say that a class potentialist system over $M \models \ZF$ is a \emph{truth potentialist system} if it satisfies the following three properties. 
\begin{enumerate}
\item $(M,\Def(M))$ is a world, where $\Def(M)$ is the collection of parametrically first-order definable classes over $M$. 
\item If $(M,\Xcal)$ is a world then it satisfies $\GB$. 
\item If $(M,\Xcal)$ is a world and $A \in \Xcal$ then $(M,\Xcal[\Tr(A)])$ is a world.
\end{enumerate}
We can modify the third condition to require truth predicates of a longer length, say of length $\mathord{<}\Lambda$. We call such a system a \emph{$\mathord{<}\Lambda$-length truth potentialist system}.
\begin{enumerate}[(1${}_\Lambda$)]
\setcounter{enumi}{2}
\item If $(M,\Xcal)$ is a world, $\xi < \Lambda$, and $A \in \Xcal$ then $(M,\Xcal[\Tr_\xi(A)])$ is a world.
\end{enumerate}
Observe that truth potentialist systems \textit{simpliciter}, as in Definition~\ref{def:truth-pot}, are the special case where we require iterated truth predicates of length $< 2$. This condition can be further modified, in the obvious way, to require truth predicates along class well-orders of unbounded length. The latter situation we'll call an \emph{unbounded truth potentialist system}.
\end{definition}

The intent of this definition is to give a modal interpretation of views such as Fujimoto's or Linnebo's. Unbounded truth potentialist systems correspond to Fujimoto's proposal, allowing arbitrary length iterated truth predicates. And $\mathord{<}\Ord$-length truth potentialist systems correspond to the class theory one obtains from Linnebo's property theory where properties are individuated iteratively along $\Ord$; see \cite{RobertsTruth}. But we have formulated things in a more general context, allowing different lengths. The point is, we see the same phenomena in the general setting, so small changes in their views won't produce a different outcome.

Before we analyse truth potentialist systems, let us briefly remark that they place restrictions on which $M$ we may consider. 

\begin{proposition}[Krajewski {\cite[p. 475]{marek-mostowski1975}}] \label{prop:worldly}
Consider $(M,\Xcal)$ a transitive model of $\GB$ with $\Tr^M \in \Xcal$. Then $M$ contains a club of ordinals $\alpha$ so that $V_\alpha^M$ is an elementary submodel of $M$.
\end{proposition}

\begin{proof}
By reflection we get club many $\alpha$ so that $(V_\alpha^M,\in, \Tr \cap V_\alpha^M)$ is a $\Sigma_1$-elementary submodel of $(M,\in, \Tr)$. But we can express that $(V_\alpha^M,\in)$ is an elementary submodel with a $\Sigma_1$-formula referring to the truth predicate.
\end{proof}
 
This result rules out, for instance, pointwise-definable models like the minimum transitive model (see \cite{hamkins:math-tea}) or models without worldly cardinals from being the sets for a truth potentialist system. Indeed, such models have no $T$-expansion for any $T$ strong enough to prove the existence of the truth predicate as a class.
\smallskip

We next present some results about smallest truth potentialist systems over a fixed $M$. First, however, let us clarify in what sense a potentialist system may be smallest among a collection of systems. One way to compare potentialist systems is by containment: if $\Afrak$ and $\Bfrak$ are potentialist systems then $\Afrak \subseteq \Bfrak$ if every world in $\Afrak$ is a world in $\Bfrak$. But this comparison doesn't get at the desired phenomenon. For instance, $\Afrak$ could consist of a single model of $\GB$ with an $n$-iterated truth predicate for each $n$, and $\Bfrak$ could be obtained from $\Afrak$ by refining its one world into many by adding $\omega$ many new worlds where the $n$-th world only has up to the $n$-iterated truth predicate. Then $\Afrak \subseteq \Bfrak$, but we wouldn't say that $\Afrak$ is smaller than $\Bfrak$.
To avoid this sort of issue: Say that $\Afrak$ \emph{covers} $\Bfrak$ if every world in $\Bfrak$ is contained in some world in $\Afrak$. If $\Afrak \subseteq \Bfrak$ and $\Afrak$ covers $\Bfrak$ then we say that $\Bfrak$ \emph{refines} $\Afrak$. A potentialist system is \emph{refined} relative to a collection of systems if it has no proper refinements within the collection. Given a collection of systems, the \emph{smallest} potentialist system in the collection, if it exists, is the refined system which is covered by every other system.

That clarification out of the way, let's see that there are smallest truth potentialist systems. We state this as three results for the three different variants of the definition.

\begin{theorem}\label{thm:smallest}
If $M$ admits a truth potentialist system then it admits a smallest truth potentialist system. This potentialist system validates $\Sfourthree$.
\end{theorem}

\begin{theorem}\label{thm:extendedsmallest}
Fix a length $\Lambda \in \Def(M)$ where $\Lambda$ is well-founded.
If $M$ admits a length $\mathord{<}\Lambda$ truth potentialist system then it admits a smallest one. 
Every axiom of $\Sfourthree$ is valid in this smallest potentialist system. Moreover if either $\Lambda \ge \omega^2$ is closed under addition $\mathord{<}\Lambda$ or $\Lambda\cdot\omega$ is closed under addition $\mathord{<}\omega^2$ then the modal validities are precisely $\Sfourthree$.
\end{theorem}

\begin{theorem} \label{thm:smallest-fujimoto}
Suppose $M$ admits an unbounded truth potentialist system whose every world is a $\beta$-model. Then $M$ admits a smallest unbounded truth potentialist system, and this potentialist system validates exactly $\Sfourthree$.
\end{theorem}

Let us remark that the condition on $M$ is satisfied if $M$ has an expansion to a $\beta$-model of $\GB + \ETR$, by considering the trivial potentialist system consisting of just that one world.
\smallskip

These three theorems come from the same construction, with slight tweaks for each case, so we will only give it once, for Theorem~\ref{thm:extendedsmallest}. The reader can easily adjust the proof to give the other two theorems.

\begin{proof}
We will build the smallest $\mathord{<}\Lambda$ truth potentialist system, call it $\Tfrak$. The key point is that $\Tr_\xi(\Tr_\eta)$ is inter-definable with $\Tr_{\eta+\xi}$, so if we build up from the base world $\Xcal_0 = \Def(M)$ we only have to add $\Tr_\xi$ for enough $\xi$. Define $\Tfrak$ to consist of worlds $\Xcal_\xi$, where $\Xcal_0 = \Def(M)$ and, more generally, $\Xcal_\xi = \Xcal_0[\Tr_\xi]$. We want to take the shortest segment of $\xi$s which makes this work. For this, we must split into cases. If $\Lambda$ is closed under addition $\mathord{<}\Lambda$ then we take all $\xi < \Lambda$ for $\Xcal_\xi$ to be in $\Tfrak$. Else, we take all $\xi < \Lambda \cdot \omega$. We claim that $\Tfrak$ is the smallest $\mathord{<}\Lambda$ truth potentialist system. 

It is immediate by construction that each world satisfies all axioms of $\GB$ except possibly the class replacement axiom. Consider now any $\mathord{<}\Lambda$-truth potentialist system $\Afrak$ over $M$. An easy induction on $\xi$ shows that each $\Xcal_\xi$ is contained in a world in $\Afrak$. This establishes that $\Tfrak$ is covered by $\Afrak$ and thus each $(M,\Xcal_\xi)$ does indeed satisfy class replacement; after all, a failure of class replacement in a smaller collection of classes for $M$ would also be a failure in a larger collection of classes. A similar induction shows that if $\Tfrak$ covers $\Afrak$ then each world in $\Afrak$ is one of the $\Xcal_\xi$. Thus, $\Tfrak$ is refined.
Finally, that we can always find iterated truth predicates in a larger world is the key observation above. To find $\Tr_\xi(A)$ for a class $A \in \Xcal_\eta$, look in the world $\Xcal_{\eta+\xi}$; this world has $\Tr_{\eta+\xi}$ and, using that $A$ is definable from $\Tr_\eta$, we can use this longer truth predicate to define $\Tr_\xi(A)$.

To see that $\Tfrak$ validates $\Sfourthree$, merely observe that $\Tfrak$ is linearly ordered and recall that linear orders validate $\Sfourthree$. To get the upper bound, we will use Theorem~\ref{thm:ratchets} and demonstrate a long ratchet for this potentialist system. Let $r_\xi$ be the assertion ``$\Tr_\xi$ exists''. Then $\seq{ r_\xi : \xi < \Lambda^* }$ gives a long ratchet for this potentialist system, where $\Lambda^*$ is either $\Lambda$ or $\Lambda \cdot \omega$ depending on which case we are in for constructing $\Tfrak$. The assumptions on $\Lambda$ ensure that $\Lambda^*$ is closed under addition $\mathord{<}\omega^2$, allowing the theorem to apply.
\end{proof}

\begin{remark}
For Theorems~\ref{thm:extendedsmallest} and \ref{thm:smallest-fujimoto} the purpose of the assumptions about $\Lambda$ being well-founded and about $\beta$-models is to allow the inductions on $\xi$ to go through. These inductions take place in the meta-theory, so we need the $\xi$ to be well-founded as seen from the meta-theory.
\end{remark}

We next address potentialist systems with worlds that satisfy $\GBC$---that is, worlds that additionally have a global well-order. We emphasise that the previous results hold for any transitive $M$, whereas for results about global well-orders we will use more specific assumptions. The two cases we will focus on is when $M$ is countable and when $M = V_\kappa$ for some inaccessible cardinal $\kappa$. These two cases correspond to two possible approaches to class potentialism; see discussion in \S\ref{bottom-up} and \S\ref{top-down}.

There are a few ways one might ensure a potentialist system gives worlds which satisfy global choice. 
If $M$ has a definable global well-order then in any truth potentialist system for $M$ every world satisfies $\GBC$. So we get truth-potentialist systems which validate exactly $\Sfourthree$ and whose every world satisfies $\GBC$. However, this puts a firm restriction on the first-order theory of $M$, namely $M \models \exists x\ V = \HOD(\{x\})$.\footnote{This is like the well-known fact \cite{mcaloon1971} that $V = \HOD$ is equivalent to having a parameter-free definable global well-order; the only difference with $V = \HOD(\{x\})$ is that $x$ gets used as a parameter in the definition of the global well-order.}
One may very well think that this extra restriction is unwarranted, and so consider the general case. Here we can get different behaviour.

Here are two approaches. First, we could start with a base world that satisfies $\GBC$. If this base world is to be as small as possible, it must be of the form $(M,\Def(M,G))$, where $G$ is a global well-order for $M$. With this approach every world in the potentialist system will satisfy $\GBC$. Alternatively, we could add a new rule saying that we can expand to a larger world to add a global well-order. Some worlds, such as the base world of definable classes, may not satisfy global choice, but you can always extend to a world satisfying $\GBC$.

For the first approach, the way to formulate this is to replace condition $(1)$ in the definition of a truth potentialist system with the following.
\begin{enumerate}[(1${}_G$)]
\item There is a world of the form $(M,\Def(M,G))$ where $G$ is a global well-order for $M$, and all worlds extend this base world.
\end{enumerate}

To start let's see that some (though not all) global well-orders are inter-definable with Cohen-generic classes of ordinals. For one direction, if $C \subseteq \Ord$ is Cohen-generic then, by density, every set is coded into the bit pattern of $C$. So we can define a global well-order by comparing where sets are first coded. Given such a global well-order we can rearrange it so that it has the ordinals, in increasing order, placed precisely on the indices where $C$ has value $1$. This rearrangement is still definable from $C$, and notice that if we have such a global well-order we can recover $C$ by looking at the indices for where ordinals appear. So the two are inter-definable. Indeed, we can say more about generic global well-orders. There is a natural class forcing to add a global well-order without adding new sets. Namely, let $\Qbb$ consist of set-sized well-orders, ordered by end-extension. By density a generic $H$ for $\Qbb$ will have all of $V$ as its domain, and it doesn't add sets because $\Qbb$ is $\mathord{<}\kappa$-closed for every $\kappa$. Given such $H$ we can define a Cohen-generic $C$ by putting $i \in C$ if and only if the $i$-th element of $H$ is an ordinal. Thus, the two forcings are forcing equivalent.

Starting with a global well-order in the base world can affect the modal validities. This phenomenon happens in both the countable case and the $V_\kappa$ case. For the sake of readability we state it as two theorems.

\begin{theorem} \label{thm:s42}
Let $M$ be a countable transitive model of $\ZFC$ and let $\Lambda$ be a limit-length well-order over $M$.
Suppose that $M$ is the first-order part of a model of $\GB + \ETR(\Lambda)$.
There are $\mathord{<}\Lambda$ truth potentialist systems, but modified to require a global well-order in the base world, over $M$ whose modal validities are precisely $\Sfourtwo$, where we allow parameters. In particular, $\mathsf{.3}$ is invalid for these systems.
\end{theorem}

\begin{proof}
Consider $(M,\Ycal) \models \GB + \ETR(\Lambda)$ where $\Ycal$ is countable.
Force over this model to add a Cohen generic $C$. If $M$ is countable such exists because there are only countably many dense sets to meet. If $M$ is a $V_\kappa$, note that $\Add(\Ord,1)^M = \Add(\kappa,1)$ is $\mathord{<}\kappa$-closed and since there are only $\kappa$ many dense sets in $\Ycal$ we can meet them to find a generic.
Gitman and Hamkins showed that tame class forcing, such as Cohen forcing, preserves $\ETR(\Lambda)$ \cite[Theorem 16]{HamkinsWoodin2018}.

Consider the truth-potentialist system $\Tfrak$ starting with the base world $(M,\Xcal_0)$, where $\Xcal_0 = \Def(M,C)$, and closing off under the requirement that any world can be extended by adding a $\xi$-iterated truth predicate relative to a parameter for any $\xi < \Lambda$. That is, $\Tfrak$ consists of worlds of the form $\Xcal_0[\Tr_\xi(A)]$ where $A \in \Xcal_0$.
We can meet this requirement to obtain a truth potentialist system, as all worlds reachable in this way are coded in $(M,\Ycal[C])$, and so all worlds must satisfy $\GB$.

We claim this truth-potentialist system has precisely $\Sfourtwo$ as its modal validities. For the lower bound, it suffices to observe it is directed: if $(M,\Xcal_0)$ and $(M,\Xcal_1)$ are two worlds, then they both are contained within $(M,\Def(M,\Tr_\xi(C)))$ for some large enough $\xi < \Lambda$. 

For the upper bound, by Theorem~\ref{thm:buttons-and-switches} it is enough to show that there are arbitrarily large families of independent buttons and switches. To do this, recall that $\Add(\Ord,1)$ is equivalent to $\Add(\Ord,\omega)$. So we can split $C$ into $\omega$ many classes $C_i$ so that the $C_i$ are mutually generic over $(M,\Ycal)$. Further, this splitting process is definable, so the $C_i$ are uniformly definable from the parameter $C$---you can take $C_i$ to be formed from the bits on the coordinates equivalent to $i$ modulo $\omega$. In particular, this means that given any $\xi,\eta < \Lambda$ , if $i \ne j$ then $\Tr_\xi(C_i)$ is not definable from $\Tr_\eta(C_j)$. This is just because $\Tr_\eta(C_j)$ is in $(M,\Ycal[C_j])$ and $C_i$ is generic over that model. 

Fix a world $(M,\Xcal)$ in this potentialist system to work inside. Let $\mu$ be the supremum of the lengths of the iterated truth predicates over the $C_i$'s which are in $\Xcal$. We will use the even coordinates $i$ for our buttons and the odd coordinates $i$ for our switches. For the buttons, let $\beta_i$ be the statement ``$\Tr_{\mu+1}(C_i)$ exists''. For the switches, let $\sigma_i$ be the statement ``if $\xi$ is the largest ordinal for which $\Tr_\xi(C_i)$ exists, then $\xi$ is even''. As a consequence of mutual genericity from the previous paragraph we get that they are independent; we can add a longer truth predicate relative to $C_i$ without affecting which truth predicates exist relative to the $C_j$ for $j \ne i$. This completes the proof.
\end{proof}

\begin{theorem}\label{thm:s42-kappa}
Let $M = V_\kappa$ for inaccessible $\kappa$ and let $\Lambda$ be a limit-length well-order over $M$.
There are $\mathord{<}\Lambda$ truth potentialist systems, but modified to require a global well-order in the base world, over $M$ whose modal validities are precisely $\Sfourtwo$, where we allow parameters. In particular, $\mathsf{.3}$ is invalid for these systems.
\end{theorem}

\begin{proof}[Proof Sketch]
The same construction goes through. The only thing to check is that we can find the desired generic $C$. First, by taking an elementary submodel of $(M,\powerset(M))$ we get $(M,\Ycal) \models \GB + \ETR(\Lambda)$ where $\card \Ycal = \card M = \kappa$. We can force over this model to add a Cohen generic $C$. This is because $\Add(\Ord,1)^M = \Add(\kappa,1)$ is $\mathord{<}\kappa$-closed and since there are only $\kappa$ many dense sets in $\Ycal$ we can meet them to find a generic $C$.
\end{proof}

It may be helpful to see an explicit example of an instance of $\mathsf{.3}$ which is invalidated by this potentialist system. Suppose we are living in a world $(M,\Xcal)$ and define $\mu$ as in the proof. Let $\varphi$ be the assertion ``$\Tr_{\mu+1}(C_0)$ exists but $\Tr_{\mu+1}(C_1)$ does not exist'' and $\psi$ be the assertion where we swap the two coordinates, namely ``$\Tr_{\mu+1}(C_1)$ exists but $\Tr_{\mu+1}(C_0)$ does not exist''. Then $\varphi$ and $\psi$ are both possible at $(M,\Xcal)$. However, if $\varphi$ is true at a world then $\psi$ is impossible at that world, and if $\psi$ is true at a world then $\varphi$ is impossible at that world, giving a failure of $\mathsf{.3}$.

\begin{remark}
This construction can also be used to give unbounded truth potentialist systems, i.e. where the iterated truth predicates can be of any length, modified to have a global well-order in the base world whose modal validities (allowing parameters) are precisely $\Sfourtwo$. In the $V_\kappa$ case this requires no extra assumption, while in the countable transitive model case this requires that $M$ be the first-order part of a model of $\GB + \ETR$.
\end{remark}

This construction also has implications for top-down potentialist systems. We state this only for the countable case, as we give a more precise calculation of the modal validities in the $V_\kappa$ case in Corollary~\ref{cor:kappa-top-down-sfourtwo}.

\begin{corollary} \label{cor:ctm-top-down-point3fails}
Fix countable transitive $M \models \ZFC$ and suppose $M$ is the sets of some model of $\ETR(\Lambda)$ for some limit length well-order $\Lambda$. Consider the top-down potentialist system consisting of all countable $\GB$-expansions of $M$. Then this potentialist system does not validate $\mathsf{.3}$, allowing parameters in formulae.
\end{corollary}

\begin{proof}
The same assertion $\varphi$ witnesses the failure of $\mathsf{.3}$. Indeed, the same construction of independent buttons and switches to get $\Sfourtwo$ as an upper bound applies.
\end{proof}

In Corollary~\ref{cor:ctm-top-down-point2fails} we will see that $\mathsf{.2}$ may also fail.
\smallskip

We now turn to the other option for getting worlds which satisfy global choice. Rather than add a condition asserting that the base world contains a global well-order, we instead add a rule allowing the addition of a global well-order. Let's first focus on the stronger assumption that the global well-order is added by forcing. That is, we add a new rule to get new worlds: if $(M,\Xcal)$ is a world then so is $(M,\Xcal[C])$ whenever $C$ is a Cohen subclass of $\Ord$ generic over $(M,\Xcal)$. 
Of course, which Cohen generics exist will depend on the choice of model. If $(M,\Xcal)$ is countable, this is just the well-known Rasiowa--Sikorski lemma that you can always meet countably many dense sets. In the $M = V_\kappa$ case, if $\Xcal$ has cardinality $\kappa$ then we can build a generic in $\kappa$ many sets using the $\mathord{<}\kappa$-closure of $\Add(\Ord,1)^M = \Add(\kappa,1)$.
We will start with the definable classes as the smallest world, and keep the old rule about being able to add truth predicates relative to extant classes. 

If $M$ is countable this is quite destructive. Specifically, adding a global well-order may kill off the possibility of adding a truth predicate whilst preserving the basic axioms of $\GB$.\footnote{The non-relativised version of this theorem can be found as \cite[Theorem 1.57]{williams:diss}.}

\begin{theorem} \label{thm:killing-truth}
Let $M$ be a countable transitive model of $\ZFC$ and let $A \subseteq M$ be a class over $M$ so that $(M,\Def(M,A)) \models \GB$. Then there is $C$ Cohen-generic over $(M,\Def(M,A))$ so that no $\GB$-expansion for $M$ can contain both $C$ and $\Tr(A)^M$.
\end{theorem}

\begin{proof}
If no $\GB$-expansion for $M$ contains $\Tr(A)^M$ then we are trivially done. So suppose we are not in this case.

We claim that from $\Tr(A)^M$ we can define a sequence $\seq{D_\alpha : \alpha \in \Ord}$ of dense subclasses of $\Add(\Ord,1)$ so that meeting every $D_\alpha$ guarantees genericity over $(M,\Def(M,A))$. This is because we can take $D_\alpha$ to be the intersection of all open dense classes definable from $A$ with parameters from $V_\alpha$. The point is, from the truth predicate we can define this sequence, because the truth predicate gives us uniform access to definability. Then $D_\alpha$ is open dense because the forcing is $\kappa$-distributive for every $\kappa$. And clearly meeting every $D_\alpha$ implies getting below every definable dense class. 

Now let's use this sequence of $D_\alpha$s to code a bad real into a generic. Fix a binary sequence $B : \Ord \to 2$ so that the set of $i$ so that $B(i) = 1$ is cofinal in the ordinals of $M$ and has ordertype $\omega$. Note that no $\GB$-expansion for $M$ can contain $B$, as $B$ reveals that $M$'s ordinals have countable cofinality. Define a sequence of conditions: Start with $p_0 = \emptyset$. Given $p_\alpha$, extend to meet $D_\alpha$, where we require the extension to be of minimal length to get into $D_\alpha$. Then add on the bit $B(\alpha)$ to get $p_{\alpha+1}$. And at limit stages, set $p_\lambda = \bigcup_{\alpha < \lambda} p_\alpha$. Note that the sequence $\seq{p_\alpha : \alpha < \lambda}$ is an element of $M$---and thus $p_\lambda \in M$---by application of the class replacement axiom in an $\GB$-expansion for $M$ containing $\Tr(A)^M$, using that any initial segment of $B$ is in $M$.
Then $C = \bigcup p_\alpha$ meets every dense class in $\Def(M,A)$.

Finally, note that if you have both the sequence $\langle D_\alpha \rangle$ and $C$, you can recover the coding points and thereby recover $B$. This is because, given these data, it is a definable property to see the shortest distance you need to extend to meet the next dense class. So if you had both $C$ and the truth predicate in an $\GB$-expansion for $M$, then you would also have $B$ in the $\GB$-expansion. This is impossible.
\end{proof}

This result immediately implies we can kill off iterated truth predicates.

\begin{corollary}
Let $M$ be a countable transitive model of $\ZFC$, let $A \subseteq M$ be a class over $M$ so that $(M,\Def(M,A)) \models \GB$, and let $\Lambda \in \Def(M,A)$ be a class well-order. Then there is $C$ Cohen-generic over $(M,\Def(M,\Tr_\Lambda(A)))$ so that no $\GB$-expansion for $M$ can contain both $C$ and $\Tr_{\Lambda+1}(A)^M$.
\end{corollary}

\begin{proof}
Assume we are in the nontrivial case where some $\GB$-expansion of $M$ contains $\Tr_{\Lambda+1}(A)^M$. Now apply the theorem to $\Tr(\Tr_{\Lambda}(A))$ to produce a generic $C$. Because $\Tr(\Tr_{\Lambda}(A))$ is inter-definable with $\Tr_{\Lambda+1}(A)$, the generic $C$ cannot be in a $\GB$-expansion for $M$ which contains this iterated truth predicate.
\end{proof}

This killing truth phenomenon rules out the possibility of getting a truth potentialist system with allowing extensions to add a generic global well-orders, in the case of working over a countable transitive model.

\begin{corollary}
Let $M$ be a countable transitive model of $\ZFC$ and $\Lambda$ be a limit length well-order over $M$ so that $M$ has a nontrivial $\mathord{<}\Lambda$ truth potentialist system. Consider the smallest $\mathord{<}\Lambda$ truth potentialist system over $M$, call it $\Afrak$. Consider the potentialist system, call it $\Bfrak$, you get by expanding $\Afrak$ by adding in all extensions of worlds in $\Afrak$ by adding a generic global well-order. Then, for any world in $\Bfrak$ which came from $\Afrak$, $(1)$ the $\mathsf{.2}$ axiom is invalid in $\Bfrak$ at every world which came from $\Afrak$, where we allow parameters; and $(2)$ there are worlds in $\Bfrak$ with classes for which it is impossible to find a truth predicate in a larger model of $\GB$ over $M$. That is, this larger potentialist system cannot be expanded to become a truth potentialist system.
\end{corollary}

\begin{proof}
Because $\Afrak$ is the smallest $\mathord{<}\Lambda$ truth potentialist system over $M$, each world in $\Afrak$ has classes of the form $\Xcal_\xi = \Def(M,\Tr_\xi)$ for some $\xi < \Lambda$. Fix $\xi$ and consider the corresponding world in $\Afrak$. Then there is $C \subseteq \Ord^M$ Cohen-generic over $(M,\Xcal_\xi)$ so that no $\GB$-expansion for $M$ can contain both $C$ and $\Tr_{\xi+1}$. Let $\varphi$ be the formula asserting that $\Tr_{\xi+1}$ exists. Then $\varphi$ is possibly necessary at $\Xcal_\xi$, since it is necessary at $\Xcal_{\xi+1}$ but is not necessarily possible, since it is impossible in $\Xcal_\xi[C]$. This establishes $(1)$ of the corollary. 
For $(2)$, the theorem tells us that $\Tr(C)$ cannot exist in any $\GB$-expansion of $M$.
\end{proof}

\begin{remark}
If, instead of only allowing extensions adding generic global well-orders, we include extensions adding any global well-order, we still get the same phenomenon. This is because such a potentialist system would include extensions by generic global well-orders, so the truth-killing theorem would still apply.
\end{remark}

For a similar reason, this result affects the modal validities for top-down potentialist systems over a countable transitive model.

\begin{corollary} \label{cor:ctm-top-down-point2fails}
Fix countable transitive $M \models \ZFC$ and consider the top-down class potentialist system $\Afrak$ for $M$ consisting of countable $(M,\Xcal) \models \GB$. If $(M,\Xcal)$ is a world in $\Afrak$ which contains a class $A$, does not contain not its truth predicate $\Tr(A)$, but does have an expansion to a world which contains $\Tr(A)$, then the $\mathsf{.2}$ axiom is invalid at $(M,\Xcal)$, allowing $A$ as a parameter. In particular, if $A$ is parameter-free definable, then $\mathsf{.2}$ is invalid at $(M,\Xcal)$ without allowing parameters.
\end{corollary}

For example, if $M$ has a $\GB$-expansion which contains $\Tr$, then $(M,\Def(M))$ is such a world.

\begin{proof}
The statement ``$\Tr(A)$ exists'' is possibly necessary but by the theorem is not necessarily possible.
\end{proof}
%% also todo: change section 6 stuff to point to this rather than being said

On the other hand, if $M = V_\kappa$ for $\kappa$ inaccessible  then this sort of destructive behaviour is impossible.

\begin{observation}\label{bringtogetherkappa}
Suppose $M = V_\kappa$ where $\kappa$ is inaccessible and consider $(M,\Xcal) \models \GB$. If $C$ is any Cohen generic class of ordinals over $(M,\Xcal)$ then there is an extension $(M,\Ycal) \models \GB$ of $(M,\Xcal)$ which contains $C$ and a truth predicate relative to any class in $\Xcal$. Indeed, we may pick $\Ycal$ to contain $\Tr_\xi(A)$ for any $\xi,A \in \Xcal$.
\end{observation}

\begin{proof}
One can construct $\Ycal$ in a minimal manner, but let's use a blunt tool. Let $(M,\Ycal)$ be an elementary submodel of $(M,\powerset(M))$ which contains $\Xcal$ and $A$. Done.
\end{proof}

%% these aren't corollaries per se? Call them something else?
\begin{proposition}\label{sfourtwokappa}
Suppose $M = V_\kappa$ where $\kappa$ is inaccessible and let $\Afrak$ be the smallest $\mathord{<}\Lambda$ truth potentialist system over $M$ for some limit length $\Lambda$. Consider a potentialist system $\Bfrak$ created by taking $\Afrak$ by adding in extensions of worlds in $\Afrak$ by adding a global well-order, then closing off by adding $\mathord{<}\Lambda$ iterated truth predicates. Then $\Bfrak$ validates exactly $\Sfourtwo$, allowing parameters. %If either $\Lambda \ge \omega^2$ is closed under addition $\mathord{<}\Lambda$ or $\Lambda\cdot\omega$ is closed under addition $\mathord{<}\omega^2$ then $\Sfourttw$ is an upper bound, allowing parameters.
\end{proposition}

\begin{proof}[Proof Sketch]
The observation implies that $\Bfrak$ is directed as a partial order---given $\Xcal$ and $\Ycal$ from $\Bfrak$ then the closure of $\Xcal \cup \Ycal$ under definability will be a world in $\Bfrak$---and so it validates $\Sfourtwo$. To show that $\Sfourtwo$ is also an upper bound, we use the same strategy as in Theorem~\ref{thm:s42-kappa}. If $(M,\Xcal)$ is a world in $\Bfrak$ containing a class $C$ which is Cohen generic over a smaller world in $\Bfrak$, then there is a bound to the length of the iterated truth predicates relative to $C$ which are in $\Xcal$. This gives the necessary space to define independent buttons and switches by looking at iterated truth predicates relative to fragments $C_i$ of $C$.
\end{proof}

Similar considerations apply to top-down potentialist systems.

\begin{corollary} \label{cor:kappa-top-down-sfourtwo}
Suppose $M = V_\kappa$ where $\kappa$ is inaccessible and let $\Afrak$ be the top-down class potentialist system for $M$ consisting of all $(M,\Xcal) \models \GB$ with $\card \Xcal = \kappa$. Then the modal validities for $\Afrak$ are exactly $\Sfourtwo$, allowing parameters. Indeed, every world in $\Afrak$ extends to a world whose validities are exactly $\Sfourtwo$, again allowing parameters.
\end{corollary}

\begin{proof}[Proof Sketch]
That $\Sfourtwo$ is a lower bound is because $\Afrak$ is directed as a partial order.
Consider any world $(M,\Xcal)$ in $\Afrak$. Pick $X \subseteq M^2$ so that $\Xcal$ is coded by $X$, meaning $\Xcal = \{ (X)_i : i \in M \}$. Then $\Ycal = \Def(M,X)$ gives a world in $\Afrak$ which extends $\Xcal$. Extend $(M,\Ycal)$ to $(M,\Ycal[C])$, where $C$ is Cohen-generic over $\Ycal$. The same strategy as before yields that the modal validities at $(M,\Ycal[C])$ are exactly $\Sfourtwo$.
\end{proof}

\section{The bottom-up approach} \label{bottom-up}

In this section we discuss what light the mathematical results from the previous sections shed on bottom-up approaches to class potentialism. Our results indicate that various bottom-up class potentialist systems might not be convergent. Why this is relevant for bottom-up approaches? Well, these begin by specifying some initial starting world (i.e. \textbf{Initial World}) and then individuate new classes over this and subsequent worlds (i.e. \textbf{Individuation}). What a lack of convergence shows is that within these systems there are `choice-points'---positions in the system where we must choose to go one way rather than another.

Some of the potentialist systems we have considered exhibit explicit failures of the $\mathsf{.2}$ axiom.
%\footnote{Whilst we do not compute the exact modal logic of our systems in which $\mathsf{.2}$ fails, this failure puts these systems on the radical branching side of the divide.
%See also the discussion in \S\ref{ufs-in-classes?} where we conjecture that certain top-down class potentialist systems validate exactly $\Sfour$.} %% kill footnote? It feels to me that it makes our point sound less forceful if we link it to a conjecture. --KJW
If we imagine progressively building up the classes in such a system, we face choices of permanent consequence---statements like ``There is a truth predicate for $A$'' (where $A$ is some class or other) can be made true (and hence also necessary), but equally can be made impossible. 
 
There are at least two ways one might react to such non-convergence. One is to view non-convergence as a substantial \textit{cost}---we want our mathematical concepts of set and class, even if modal, not to contain these choice points both for philosophical cohesiveness and mathematical expedience. Another way to view them is as \textit{interesting} but not any special cost---they indicate structural properties of the relevant potentialist system (and perhaps the underlying concepts), but this feature is unproblematic. Whether or not they are taken to be a cost or merely a point of interest may depend somewhat on one's philosophical outlook. This highlights some important choice points for the bottom-up theorist.
 
 %As we remarked earlier (\S\ref{systems}) this non-convergence might be viewed as a cost or a feature, depending upon one's outlook. So, how does the `desideratum' of convergence affect the class-theoretic potentialist?

\textbf{If we want convergence, smallest systems are good.} We know that various kinds of classes can interfere with the addition of truth predicates, resulting in non-convergence. However, the \textit{smallest} systems (in the sense outlined in Theorems \ref{thm:smallest} and \ref{thm:extendedsmallest}) do \textit{not} exhibit branching, and validate $\Sfourthree$. (Indeed, some of these systems validate precisely $\Sfourthree$.) Thus, if we want to ensure non-branching, a good way to do so is to insist that we consider smallest potentialist systems. This insistence is natural with a generative view of classes, as from \cite{Fujimoto2019a} or derived from \cite{Linnebo2006a}.
Having a smallest potentialist system amounts to being able to say: ``This is the minimal amount needed to realise these generative conditions''. So, for example, Theorem~\ref{thm:smallest-fujimoto} gives that there is a smallest potentialist system which corresponds to a modal reworking of Fujimoto's liberal predicativism. 
%\footnote{There is a small caveat here. That theorem has an assumption about having worlds which are correct about which classes are well-founded. This is free on the $V_\kappa$ approach---any model of uncountable cofinality is a $\beta$-model---and this matches the actualist view of sets which motivates that approach. For the countable transitive model approach, motivated by set-theoretic multiversism, this is an extra assumption.} %% say more??? --KJW
Similarly, Theorem~\ref{thm:extendedsmallest}, with the length $\Lambda = \Ord$, says that there is a smallest potentialist system corresponding to Linnebo's property theory.\footnote{We thank
%[name removed for blind review]
\O{}ystein Linnebo for emphasising this point.}

How well does the insistence on smallest systems mesh with the two views? On the one hand, the predicativist who is only interested in adding truth predicates may have some motivation to take this position. (The case where other predicates are allowed significantly complicates things for them, and we consider this situation below when discussing global choice.) Whilst it is somewhat contingent upon the nature of the space of possible language expansions, it seems reasonable to assume that when we introduce individual new truth predicates into our language we do not thereby introduce further predicates beyond what is required by (i.e. definable in) the expansion. In this case, one clearly obtains the smallest such system any time one introduces a new truth predicate.

For property theories, we can simply note that the generation of properties is (by construction) limited to entities definable in a specific way. The new properties available at each additional stage are those whose application relation is definable over previous stages. Roberts \cite{RobertsProperties,RobertsTruth} has analysed these systems as corresponding to what classes are definable from longer and longer iterated truth predicates. Formulated in the potentialist perspective, this is looking at the smallest truth potentialist systems constructed in \S\ref{truth-potentialism}, where the lengths of iterated truth predicates one allows may depend on other philosophical commitments. %%One can view this iterated definability as a class-theoretic version of G\"odel's constructible universe.\footnote{Indeed, this can be made mathematically precise; see \cite[Section 3.2]{williams:diss}.} %% is there another place where this is? I only know of where it's in my dissertation, but I cannot imagine I was the first to make this connection --KJW
%% nevermind I'm just cutting this altogether we don't need to talk about L at all it's just confusing

{\bf Wider systems can admit branching.} A critical point, in contradistinction to the foregoing, is that for larger systems we \textit{do} get branching over a countable transitive model. If a certain kind of richness is needed or wanted by the bottom-up theorist, and in particular if they wish to transcend smallest systems, we often can get branching in those systems. If one thinks that branching is a cost (say because it indicates a kind of non-inevitability in how the classes unfold), then such a richness assumption seems like a dangerous desideratum.

The systems we considered can be seen as larger in \emph{width}, not larger in \emph{height}. Whilst our intention here is to bring to mind the familiar width versus height distinction for sets, the notion is different here, since all classes have the same height in the sense of ordinal rank. Here, height refers to the lengths of truth predicates. The wider systems we considered were those that allowed the addition of generic global well-orders (equivalently, Cohen-generic classes of ordinals). Genericity ensures that adding these does not increase the height of a world.
If you drop the requirement for genericity then global well-orders can add height. Given $(M,\Xcal) \models \GBC$, you can find a global well-order in $\Xcal$ which codes any given $A \in \Xcal$, say by placing $A$ on the even indices in the order. So you could add, over the definable classes, a global well-order which codes a very long truth predicate, or any other arbitrarily complicated class.

One might view these two observations (concerning smallest and larger systems) as a point in favour of the pictures articulated by the versions of bottom-up truth-theoretic potentialism  we have considered here. If one views branching as a cost, one way to ensure branching is avoided is to consider smallest potentialist systems. As it turns out, this is precisely what the truth-theoretic versions of liberal predicativism and property theory motivate (since they just involve adding truth predicates and closing under definability). Thus for these views there is conformity between desirable properties of the potentialist system and the details of what the relevant philosophical view motivates. As we shall see, however, allowing class forcing greatly complicates the issue.

{\bf Global choice and class forcing are problematic.} A theme in some of our results is that global choice is problematic (or at least raises several questions) in the class-theoretic potentialist context. One possibility is that we could require the global well-order to be there from the start. For many of our potentialists (e.g. the predicativist and the property theorist) the base world contains just the definable classes. To insist then that the base world contains a global well-order is just equivalent to the base world having a first-order definable well-order of the universe. This has serious first-order consequences, in particular it is equivalent to $\exists x\ V=\HOD(\{x\})$. (See the discussion in \S\ref{truth-potentialism}.) We might be suspicious of our class-theoretic commitments delivering such strong set-theoretic consequences, especially given that we are thinking of building the classes \textit{after} (in the class-theoretic potentialist's modal sense) the universe of sets has been constructed.

Another possibility is that the global well-order is generic, in the sense of class forcing. (See the discussion in \S\ref{truth-potentialism} of how to force to add a global well-order.) A potentialist might not want a global well-order which codes complicated undefinable classes, and instead want it to be ``random'' with respect to the definable classes. This amounts to asking it to be generic; extending by a generic global well-order is adding the well-order and, through the use of forcing-names, closing off under definability from the well-order and classes in the ground model. This has the interesting feature that such a well-order cannot (unlike truth predicates) be viewed as the introduction of a \textit{unique} class with a certain property---the introduction of one Cohen-generic necessitates the addition of many. However, it can still be viewed a individuation of a certain kind---we ask for a generic object and then close under definability. We leave it open whether such `indeterminate individuation' should be acceptable to the kinds of bottom-up theorist we've considered.\footnote{This is not unique to the class forcing case---it applies to set-forcing too (e.g. in the addition of a single Cohen real over a model of $\mathsf{ZFC}$).} %We should also identify that these kinds of `indeterminate individuation' aren't \textit{obviously} illegitimate in general. Consider a case where we are considering the potentialism of extending a given finite line $l$ in a Euclidean space. Any time we extend $l$, we introduce a dense-order of other `shorter' lines (i.e. possible worlds) that we could have extended to. But there doesn't seem to be anything obviously wrong with considering such a potentialism.

If we are allowed to add such a generic well-order then, as we noted in Theorem \ref{thm:killing-truth}, there is no \textit{prima facie} guarantee that the generic not be a bad one which kills off the addition of truth predicates. One response to this predicament is to require worlds to satisfy a theory which ensures the existence of all desired truth predicates. For example, if our first world satisfies $\GB+\ETR$ then all the truth predicates are already there, and so a bad `truth-killing' well-order cannot also be there. This, however, incurs the cost that the truth-theoretic potentialism is essentially trivial---all paramaterised truth predicates are there from the get-go. This contravenes the basic set-up of the liberal predicativist and property theorist and so would necessitate some revision of these positions.
%\footnote{Moreover, there is the worry that the killing truth phenomenon can be generalised higher, to prevent expansions in worlds satisfying stronger theories. We leave fleshing out this worry to future work.}

If the global well-order is to be neither definable nor generic, then what is it to be? It would be overly hasty to claim those as the only two possibilities, but the other possibilities of which we know strike us as artificial.\footnote{See the earlier point that any class can be coded into a global well-order.}
The explication of further possibilities would be a useful contribution to our understanding of classes, and we leave it to further work.

And of course, one possibility available to the class potentialist is to simply give up on having a global well-order. One might think this is natural for those bottom-up approaches based on liberalised forms of definability. Choice principles for sets draw motivation from a combinatorial view of sets, so if one has a non-combinatorial view of classes one might expect to not have choice principles for classes.\footnote{We thank
%[name removed for blind review]
\O{}ystein Linnebo for pressing us on this point.} 
Interestingly, for some of the bottom-up approaches we consider, this is not so immediate.

Let's start by noting that a class potentialist who wishes to prohibit global well-orders pays a mathematical cost. Namely, their class theory is unable to directly support the mathematical work which requires global choice. For example:

\begin{enumerate}
\item In the study of determinacy of class games, it plays an essential role in moving from quasi-strategies to strategies. For example, the equivalence of $\ETR$ and clopen class determinacy \cite{gitman-hamkins2016} requires global choice.

\item The standard arguments to prove some properties of the surreal numbers, such as them forming a universal ordered field, go through global choice.\footnote{See \cite[\S9]{ehrlich2012}. Note however that it remains an open question whether global choice is necessary for this result; see the discussion on this MathOverflow question \cite{MO:hamkins2017}.}
\end{enumerate}

On the one hand, these examples are somewhat niche when compared to the large number of equivalences with standard (set) $\AC$ and many fundamental statements of mathematics. 
But on the other hand, the equivalences with $\AC$ also rely on creatures of the higher infinite. For example, $\AC$ is equivalent over $\ZF$ to the assertion that every vector space has a basis \cite{blass1984}. But the vector spaces used in this equivalence are not the familiar ones from undergraduate linear algebra, and are necessarily quite large in cardinality. It is consistent that $\AC$ fails and yet every vector space in, say, $V_{\omega + \omega}$ has a basis. It strikes us as artificial to be sceptical of uses of choice at the level of classes but not be sceptical of uses at the level of sets of high von Neumann rank, though we acknowledge that there is more debate to be had.
 
One might think that the ability to have some worlds where we have the required classes to nicely interpret the reasoning of set theorists is, \textit{ceteris paribus}, a plus. The specific case of liberalised predicativism provides an example here---part of what is at issue for them is to provide an account of classes that is predicativist in spirit, but nonetheless yields enough strength to be able to interpret parts of set theory that use non-definable classes. Fujimoto, for example, wants to ``provide a mathematical framework in which\dots widely accepted and/or mathematically fruitful uses of classes can be meaningfully expressed and implemented'' \cite[p. 217]{Fujimoto2019a}. %writes:
%\begin{quote}
%The second desideratum is (ii) an appropriate interpretation of classes should provide a mathematical framework in which (or, at least, should be compatible with mathematical presuppositions under which) widely accepted and/or mathematically fruitful uses of classes in set theory can be meaningfully expressed and implemented. \cite[p. 217]{Fujimoto2019a}
%\end{quote}
Given that there are mathematical uses for global choice, %and (b) Fujimoto's liberal predicativist is committed to a framework for  `mathematically fruitful' uses of classes, 
it thus seems that there's some pressure to accept the possible existence of global well-orders for Fujimoto. Insofar as other approaches have similar naturalistic assumptions, they're also subject to similar pressures.\footnote{Of course one could always reject the use of such classes as `legitimate' parts of mathematics, our only contention is that it is a cost to do so.}

Given this naturalistic pressure, it's sensible to consider how the possible existence of global well-orders meshes with the philosophical positions we've considered. For the liberal predicativist, we should note that the existence of a global well-order is an assertion about the existence of a single class and does not ascribe any complicated global structure to classes as a whole. Fujimoto writes:

\begin{quote}
Liberal predicativism does well in most cases, when it comes to an axiom asserting something about a specific type of class... In contrast, if an axiom asserts something strong about the entire structure of classes or the totality of classes, liberal predicativism might be faced with a difficulty. \cite[p. 225]{Fujimoto2019a}
\end{quote}

Fujimoto is claiming here that the `bad' principles (from the liberal predicativist perspective) are those that apply global structure to the classes. Such bad examples include impredicative comprehension, where the use of unbounded quantifiers (and their alternations) can result in complex global structure. By contrast, the introduction of a truth predicate (a canonical example of a `good' class for the liberal predicativst) simply involves introducing a single class and then closing under definability. Given this characterisation, there is at least \textit{prima facie} reason for the liberal predicativist to accept the possible existence of a global well-order---it is an assertion about a single class that also meshes well with the liberal predicativist's desire for mathematical freedom. If one wishes to reject the possibility of such classes for the liberal predicativist, reasons should be given as to why. We do not rule out that there may be such principles, but we do not see any obvious candidates. Similar remarks apply to the addition of other class-forcing generics---these are assertions about the introduction of a single class (and then closure under definability), and so seem acceptable from the liberal predicativist perspective.

The situation is different for the property theorist. As noted earlier, their classes are essentially those obtained by iterating the definability predicate---equivalently, looking at what is definable from longer and longer iterated truth predicates. 
The property theorist thus rules out non-convergence by keeping a strict control on the classes that could exist. One might, of course, view this as a cost---especially if one wants to enforce as few restrictions as possible on the classes that one is allowed to form.

\textbf{The choice of model theory affects convergence.} Many of our results showing non-convergence required the use of countable transitive models. For example, Theorem \ref{thm:killing-truth} depended on adding a generic that, once we introduce the relevant truth predicate, encodes a cofinal sequence in the ordinals. In getting this generic, however, we assumed that the ground model is countable---we view the model externally as countable and talk about the ways bits can be encoded into a particular countable sequence. One might object: For many species of class-theoretic potentialist $V$ is uncountable, and so there is no such generic. For example, the class-theoretic potentialists we've presented who are also set-theoretic actualists accept this. In fact, out of the views we've considered it's only some brands of set-theoretic multiversist/potentialist who accept that $V$ can be made countable.

This contrasts sharply with the version of the model theory where we consider potentialist systems over $V_\kappa$ for $\kappa$ inaccessible. For, Observation \ref{bringtogetherkappa} says that the killing truth phenomenon does not arise for such $V_\kappa$---we can always add iterated truth predicates after adding a global well-order. And Proposition \ref{sfourtwokappa} says that there is a natural system obtained by adding a global well-order and then closing off under the addition of iterated truth predicates that validates $\Sfourtwo$. If one thinks that such a $V_\kappa$ is a better model of the universe than a countable transitive model---say because one is a set-theoretic actualist but class-theoretic potentialist---then one can avoid the killing truth phenomenon (and indeed justify $\Sfourtwo$) by adopting the appropriate model theory. But this shows that real care must be taken in providing a model-theoretic set-up; proceeding with a countable transitive model and no restriction on the generics allowed (an oft-made move) will result in (possibly unpleasant) unintended artefacts of the models employed appearing in one's conception of class-theoretic potentialism.

What emerges from this discussion is that there is the following tension at the heart of of bottom-up approaches. If we (i) regard non-convergence as a cost, (ii) want to allow the addition of truth predicates, and (iii) wish to allow unrestricted addition of generics (as with countable transitive models), then we have a problem. The property theorist resolves this issue by rejecting (iii), as does the advocate of the model theory using an inaccessible rank $V_\kappa$ instead of a countable transitive model. This problem is thus a pressing issue (only) for the predicativist who accepts the use of countable transitive models, and we will suggest a different solution via additional modal principles (also rejecting (iii)) in \S\ref{more-modal-principles}. 

\section{The top-down approach} \label{top-down}

%nownow Neil up to here

In this section we discuss what the mathematical results from \S\ref{truth-potentialism} say about top-down approaches to class potentialism, and in particular the interplay between what is satisfied on these pictures, \textbf{Referential Indeterminacy}, and \textbf{Interrelation of Interpretations}. First let's make clear just what assumptions are needed for the formal results to apply.

Our formalisation for top-down approaches is to fix a model $M \models \ZFC$ and consider a potentialist system consisting of expansions of $M$ to a model of $T$, where $T$ is a class theory. Corresponding to the set multiversist view we consider countable transitive $M$ and all countable $T$-expansions for $M$. Corresponding to set universism we consider $M = V_\kappa$ for $\kappa$ inaccessible and $T$-expansions for $M$ of size $\kappa$.  Results from \S\ref{truth-potentialism} tell us something about these potentialist systems, making some assumptions about $T$ and $M$. Two main tools were used in \S\ref{truth-potentialism}: truth predicates, and class forcing. We need both to be applicable. 

Let's discuss truth predicates first. Proposition~\ref{prop:worldly} tells us that asking to have any world with a truth predicate puts a limitation on the choice of $M$. For the $V_\kappa$ case it is already implied by $\kappa$ being inaccessible and for the countable transitive model case it amounts to requiring that $M$ satisfy a certain (second-order) reflection principle, and such principles are commonly taken to give basic properties of the universe of sets. Accordingly, it has negligible cost to assume $V$ satisfies these properties.

More substantively, these tools do not apply to any choice of $T$. The results in \S\ref{truth-potentialism} were stated in terms of iterated truth predicates. There is a limit to how far this generalises. If $T$ outright proves the existence of iterated truth predicates of any length---that is, if $T$ proves $\ETR$---one cannot have a nontrivial truth potentialist-like system whose worlds are models of $T$. 

Let's now discuss forcing. As discussed in \S\ref{truth-potentialism} adding a Cohen-generic class of ordinals adds a global choice function. This puts a restriction on the worlds allowed---if one thinks that such generics can be added, one cannot hold that Global Choice fails at every world.
The assumption that generics can be added also rules out the inclusion of axioms that limit the classes by definability.\footnote{Here's an illustrative toy example. Let $T$ be $\GB$ plus the assertions that length $n$ iterated truth predicates exist for any finite $n$ and that every class is definable from some $\Tr_n$. Any nontrivial forcing adds a generic not in the ground model and thus not definable from any $\Tr_n$, thus destroying $T$. So the results in \S\ref{truth-potentialism} do not apply to $T$.} % Indeed, $T$ is categorical over a fixed transitive $M$: given a transitive model $M$ of $\ZFC$ there is at most one $T$-expansion for $M$.}

In sum, what we need from $T$ for the results of \S\ref{truth-potentialism} to apply is: (i) $T$ does not limit the classes by definability, (ii) $T$ does not prove that arbitrary iterated truth predicates exist, (iii) $T$ does not have a principle implying global choice is false. These results illuminate in a concrete way how the indeterminacy of reference underlying top-down potentialism might manifest.
%then Theorems~\ref{thm:s42} and \ref{thm:killing-truth} both give information about the top-down potentialist system for $T$. Namely, Theorem~\ref{thm:s42} gives failures of $\mathsf{.3}$ in this system, and Theorem~\ref{thm:killing-truth} gives failures of $\mathsf{.2}$ in this potentialist system.
%% in progress

%For such a potentialist,  %Exactly as in \S\ref{bottom-up}, weak base theories (like $\GB$) seem to correspond to a conception of class that is radically divergent, it is just that in this context this radical divergence is underwritten by \textbf{Referential Indeterminacy} and \textbf{Interrelation of Interpretations} rather than what is built up from \textbf{Initial World} and {\bf Individuation}. How does this observation relate to our motivations for the top-down approach?

\textbf{Radical branching in the set-theoretic multiverse.} Let's examine the set-theoretic multiversist first. Recall that she regards talk using some class theory $T$ as just more set-theoretic mathematics up for reinterpretation. Fixing some appropriate $M$ in the multiverse, each $M'$ which has $M$ as an element and thinks it's countable will have a conception of what the $T$-class-potentialist system over $M$ looks like (for some reasonable $T$). For $T$ near the level of $\GB$ (and appropriate $M$) we get failures of the $.2$ and $.3$ axioms, indicative of strong branching; see Corollaries~\ref{cor:ctm-top-down-point3fails} and \ref{cor:ctm-top-down-point2fails}.

Some multiversists (e.g. \cite{Hamkins2018a}) distinguish between the extreme branching of $\Sfour$ and the `inevitability' of $\Sfourthree$ as well as the `convergence' implied by $\Sfourtwo$. The argument proceeds from the mirroring theorems for systems containing $\Sfourtwo$---whilst one could use the modal theory, the mirroring theorem guarantees that a modal-free theory can be used. Hamkins argues that theories with $\Sfourtwo$ are thus `quasi-actualist'---practically speaking nothing of substance hangs on whether we think of the universe modally or non-modally. For a view to be \textit{strongly potentialist}, one might think, non-convergent branching possibilities are required (and for this non-convergence to show up in the modal validities). There are many such varieties of multiversist-inspired set-potentialism on offer; rank-extension potentialism (discussed earlier) is one where the modal validities are $\Sfour$. Critically though, known examples of reasonable kinds of potentialism with branching possibilities are limited to non-well-founded models of set theory, potentialisms with only transitive models are generally $\Sfourtwo$ or stronger (e.g. both forcing and countable transitive model potentialism have $\Sfourtwo$ as their modal validities\footnote{See \cite{HamkinsLoewe2008a} and \cite{HamkinsLinnebo2018a}.}). If one then accepts (as many do) that we have an absolute understanding of well-foundedness, and that intended set-theoretic universes are all transitive/well-founded, then one seems to be (currently) bound to `quasi-actualist' potentialist systems satisfying $\Sfourtwo$ and a mirroring theorem. Not so for the multiversist-inspired class-theoretic potentialist. Here, the $\GB$-class-potentialist systems have failures of $\mathsf{.2}$ and $\mathsf{.3}$ and hence no mirroring theorem (using standard technology). These potentialist systems are the first to our knowledge systems of set theory to exhibit non-convergent branching even when we restrict to well-founded models.\footnote{Note that such occur in the context of second-order arithmetic; see \S\ref{ufs-in-classes?}.}

\textbf{Plural indeterminacy and impredicativity.}  For the theorist who holds that plural quantification is indeterminate, the situation is subtle. On the one hand, plural logic (in its standard formulation) contains all impredicative instances of the plural comprehension scheme and indeed the Henkin interpretations for plural logic obtained by \cite{FlorioLinnebo2016a} all satisfy it (they restrict to what they call \textit{faithful} models---those that satisfy every instance of the comprehension scheme). This can then be leveraged to provide an interpretation of $\MK$ class theory (as in \cite{Uzquiano2003a}). This interpretation can be carried through whether or not the range of the plural quantifiers is determinate, if it is indeterminate but nonetheless every legitimate interpretation satisfies the impredicative plural quantification scheme, this impredicativity extends immediately to obtain the impredicative class-theoretic comprehension scheme of $\MK$ within each world. As noted earlier, $\mathsf{MK}$ violates the presuppositions required to make our arguments go through since it trivialises truth-theoretic potentialism by  implying the existence of arbitrarily iterated truth predicates. Thus, our results do not have much to tell the advocate of this kind of top-down class-potentialism. To say more, further results are needed about the $\MK$-class-potentialist system, and we leave this as an open question; see the discussion around Question~\ref{q:mk-s4} for fuller details.

Though one might hold that the indeterminate plural interpretation yields $\MK$ on the basis of the `standard' conception of the logic, this is controversial. As \cite{FlorioLinnebo2016a} note, often  the impredicative plural comprehension scheme is motivated by the assumption that every non-full Henkin semantics for the plural quantifiers is unintended.\footnote{See \cite{FlorioLinnebo2016a} for further references.} 
%For example Hossack writes the following regarding the determinateness of plural quantification:
%\begin{quote}
%Plural set theory has no non-standard [in the sense of non-full Henkin semantics]
%models, so the indeterminacy problem does not arise for pluralism. 
%\sloppy \cite[p. 440]{Hossack2000a}\footnote{Examples of this kind can be multiplied. For a review of the literature see \cite{FlorioLinnebo2016a}.}
%\end{quote}
Such an assumption leads immediately to the impredicative comprehension scheme for plurals and, by extension, classes (see \cite[pp. 76--77]{Uzquiano2003a} and \cite[\S{}3.2]{Lewis1991a}). % as Uzquiano notes (referring to \S3.2 of \cite{Lewis1991a}).

%\begin{quote}
% To the extent to which one accepts unrestricted plural quantification over sets as unproblematic, one will be moved by what David Lewis refers to as the evident triviality of plural comprehension, and thus one will accept all instances of plural comprehension as true. After all, one may explain, in order for an instance of comprehension to be false, there must be a formula $\phi$ such that it is neither the case that no sets satisfy it nor is it the case that some sets satisfy it. But this could never be the case. \cite[pp. 76--77]{Uzquiano2003a}
%\end{quote}

Of course, if we allow plural quantification to be indeterminate then we have an immediate response---an instance of a formula $\phi$ in the plural comprehension schema might be neither true nor false of some sets in virtue of there being some interpretations of the plural variables on which it is true, and other interpretations on which it is false. For example, consider the following sentence:

$$\phi(x) = \text{``} x=x \text{ and the Class Fodor Principle}\text{''}$$

\vspace{0.3cm}

Whether the Class Fodor Principle holds is independent of $\MK$ \cite{GHK2019}.
If we do not assume that quantification is determinate, and there are worlds in which Class Fodor holds and others in which Class Fodor fails, then it is neither the case that no sets satisfy $\phi$ nor is it the case that some sets satisfy $\phi$---in some worlds $\phi$ picks out $V$ and in others it picks out $\emptyset$. %\footnote{This example generalises beyond just $\GB$. E.g. if $\MK$ is our base theory then instead of asking whether truth exists we should ask whether there is a class coding an $\MK$-expansion for $V$, an assertion not decided by $\MK$.}
%Throwing in impredicative plural quantification at the start simply prejudices the debate in favour of $\MK$, and the water becomes much muddier once we allow many different interpretations of the range of the plural variables.

Nonetheless, the view that Lewis' thought about pluralities is somehow \textit{part of our conception of pluralities}  is tempting. Even if we think the quantification is indeterminate, we might think that \textit{within a world} his intuition should motivate us to accept impredicative plural quantification over that world, yielding $\MK$ locally. The thought then that this reasoning should apply schematically to \textit{every} world---thereby trivialising truth-theoretic potentialism---merits further scrutiny. Here is not the place to adjudicate these difficult issues concerning the relationships between the philosophies of plural quantification and mathematics. However, these observations point to a substantial philosophical issue: There is a critical choice point in the selection of theory for the believer in the indeterminateness of plural quantification who is happy using countable transitive models. Acceptance of the impredicative plural comprehension scheme despite indeterminacy has immediate mathematical ramifications, not just regarding what non-modal statements of class theory are true at every world but also  the nature and validities of the relevant potentialist systems.

\textbf{The underlying ontology affects branching.} Let's assume that one does accept the plural interpretation, but allows that weak theories (like $\mathsf{GB}$) can be thereby motivated. Let's also assume, however, that one is a set-theoretic \textit{actualist}. For such a theorist, we contend, the model theory that uses some inaccessible $V_\kappa$ is much better for representing the class-theoretic potentialism on offer, certainly in comparison to a countable transitive model. As we noted in Corollary~\ref{cor:kappa-top-down-sfourtwo}, the corresponding top-down potentialist system validates $\Sfourtwo$ even if we consider an interpretation with a global well-order added by a Cohen-generic class, we can always consider interpretations that have a truth predicate for this class. So the divergence does not occur as it does for the set-theoretic multiversist who accepts that \textit{any} universe can be made countable. The different model theories suggested by the different underlying motivations thus have a profound influence upon the nature of the relevant class-theoretic potentialism.

\section{Conclusions and further directions}\label{conclusion}

In this paper we've argued that there are natural interpretations of class talk over a fixed domain of sets that yield \textit{potentialisms} of different kinds. We've also argued that it makes sense to divide these pictures into two kinds: bottom-up approaches begin with some fixed stock of classes and then individuate new ones, and top-down approaches take the modal variation of classes to arise from referential indeterminacy and the ways possible sharpenings of the ranges of the class variables relate to one another. We've proved several results about potentialist systems, in particular exhibiting failures of the $\mathsf{.2}$ and $\mathsf{.3}$ axioms for potentialist systems corresponding to weak theories of classes. Finally, we've discussed some philosophical payoffs of these results for the various bottom-up and top-down approaches.

This is very far from the end of the story, however, and we hope to have merely made a first-step in discussions about class-theoretic potentialism and possible responses to the challenges we have outlined. For this reason, we raise several open questions that may be of interest to others wishing to pursue this line of research. %This will make the conclusion slightly longer than is usual, but we feel that identifying salient problems is important and hope that the reader will indulge us.
%% you know what also makes the conclusion slightly longer? Asking the reader to indulge us

\subsection{Additional modal principles?} \label{more-modal-principles}

As discussed in \S\ref{top-down}, the failure of $\mathsf{.2}$ for top-down potentialism for weak class theories is particularly destructive, being witnessed by a world which cannot be further extended to add in a certain truth predicate. A top-down potentialist may very well think this catastrophic world is an artifact of the formalisation, one which does not occur in the real multiverse of classes. Her task then is to explain why this phenomenon does not occur and formulate principles prohibiting these worlds. One way to do so is in the model theory---adopting a model theory that uses inaccessible $V_\kappa$ instead of countable transitive models.

A different (but related) approach is to provide additional \textit{modal} axioms, going beyond just the resources of class theory. For instance, the following modal principle manifestly rules out the killing truth phenomenon:
\[
\necessary \forall X \possible \exists Y (\text{``}Y \text{ is a truth predicate for } X \text{''}).
\]
It is easy to formulate versions of this for iterated truth predicates. And one could consider yet more modal principles to express properties of the true multiverse of classes.

Examples of this kind already exist in the case of the set-forcing potentialist. For example, \emph{maximality principles}, assertions of the form $\possible \necessary \phi \impl \phi$, have been considered in the context of set forcing potentialism; see e.g. \cite{Hamkins2003a} and \cite{HamkinsLinnebo2018a}.
An example of a different flavor, one closer in motivation to what we give here, can be found in
\cite{Steel2014a} (with subsequent development by \cite{MaddyMeadows2020a}). Steel is investigating a multiversist framework arising from set forcing. Given a countable model of set theory, it has pairs of Cohen extensions which do not \emph{amalgamate}---there is no outer model which contains both Cohen extensions as submodels \cite{Mostowski1976a}. To exclude this phenomenon, Steel includes an axiom asserting that models in the generic multiverse always amalgamate. His principle is in fact higher-order, referring to worlds as objects, not just to what is true of sets within each world. And one could also consider higher-order principles in the context of class potentialism.

We leave the consideration of these higher-order or modal principles to future work.

\subsection{An analogy to second-order arithmetic, and universal finite sequences?} \label{ufs-in-classes?}

A potential area for further study concerns the analogy between the use of classes in the contexts of second-order arithmetic and set theory. Predicativism in mathematics often takes the totality of natural numbers as given, with the predicatively-given ``classes'' then being sets of natural numbers, e.g. \cite{FefermanHellman1995a,FefermanHellman2000a}.
There has been work addressing to what extent results about predicativism over $\omega$ generalise to predicativism over $V$---see e.g. \cite{fujimoto2012,sato2014}.
Similar to how it was formalised in the set theoretic context, one could formalise potentialism over $\omega$ by considering potentialist systems of $\omega$-models of second-order arithmetic. 
To what extent does the mathematical and philosophical work about class potentialism carry over to the arithmetic context?

We also wish to mention a question arising from the analogy going in the other direction. Here, $\Zsf_2$ is the theory of second-order arithmetic with full impredicative comprehension and $\Sigma^1_\infty$-$\AC_0$ is the choice schema for definable collections of sets; consult \cite[Chapter VII.6]{simpson:book} for a definition. If $(\omega,\Xcal)$ and $(\omega,\Ycal)$ are models of second-order arithmetic then $\Xcal$ is a \emph{$\beta$-submodel} of $\Ycal$, written $(\omega,\Xcal) \subseteq_\beta (\omega,\Ycal)$, if $\Xcal$ is a subset of $\Ycal$ and the two models agree about which of $\Xcal$'s relations are well-founded.

\begin{theorem}[Hamkins--Williams]
Let $T$ be a computably axiomatizable extension of $\Zsf_2 + \Sigma^1_\infty$-$\AC_0$. Then the modal validities of the potentialist system consisting of countable $\omega$-models of $T$ ordered by $\subseteq_\beta$ are precisely $\Sfour$, whether or not we allow parameters in formulas.
\end{theorem}

\begin{proof}[Proof sketch]
It is well-known that $\Zsf_2 + \Sigma^1_\infty$-$\AC_0$ is bi-interpretable with $\ZFm + V = H_{\omega_1}$, the assertion that every set is countable.  Given $(\omega,\Xcal)$ a model of arithmetic call the corresponding model of $\ZFm$ its \emph{companion model}. Such companion models must be well founded beyond $\omega$, and if $(\omega,\Xcal)$ is a $\beta$-submodel of $(\omega,\Ycal)$ then the companion model of $(\omega,\Xcal)$ is end-extended by the companion model of $(\omega,\Ycal)$. 
Consider now the potentialist system consisting of these countable, $\omega$-standard companion models, ordered by end-extension. Up to coding this is the same as considering the potentialist system of countable $\omega$-models of $T$ directly. An instance of 
\sloppy \cite[Theorem 6]{hamkins-williams2021}
yields that this potentialist system admits a universal finite sequence and thus its modal validities are precisely $\Sfour$.\footnote{Theorem 6 is phrased in a general form. See the discussion at the beginning of \S4 for why it applies to extensions of $\ZFm$.}
\end{proof}

Does this theorem generalise to the set theoretic context? More precisely:

\begin{question} \label{q:mk-s4}
Let $T$ be $\MK$ plus Class Bounding and let $M$ be a countable transitive model of $\ZFC$ which has a nontrivial top-down class potentialist system consisting of $T$-expansions for $M$. Does the potentialist system consisting of countable $T$-expansions for $M$ ordered by the substructure relation\footnote{For models of class theory with the same sets every submodel is a $\beta$-submodel for free.} 
have $\Sfour$ as its modal validities?
\end{question}

A positive answer to this question would imply that the fundamental branching phenomenon for top-down potentialism for weak theories also occurs for very strong theories. % say more?

%bibliography materials
\nocite{Ewald1996b} %a so this shows up for the "Appears in" for Zermelo1930
\bibliography{18nabbib18,local}
\bibliographystyle{apalike}

\end{document}